\documentclass[a4paper, 12pt]{amsart}
\usepackage{amsmath}
\usepackage{amssymb}
\usepackage{amsthm, enumerate}
\usepackage[top=20truemm, left=17truemm, right=17truemm]{geometry}

\usepackage{xcolor}
\usepackage{soul}

\usepackage{amscd}
\usepackage{cleveref}
\usepackage{tikz-cd}

\newcommand{\C}{\mathbb{C}}
\newcommand{\R}{\mathbb{R}}
\newcommand{\Z}{\mathbb{Z}}
\newcommand{\N}{\mathbb{N}}

\newcommand{\bK}{\mathbb{K}}

\newtheorem{dfn}{Definition}[section]
\newtheorem{thm}[dfn]{Theorem}
\newtheorem{lem}[dfn]{Lemma}
\newtheorem{rem}[dfn]{Remark}
\newtheorem{cor}[dfn]{Corollary}

\newtheorem{prop}[dfn]{Proposition}
\newtheorem{ex}[dfn]{Example}

\usepackage{graphicx}

\usepackage{tikz}
\usetikzlibrary{arrows}

\title[Spanier--Whitehead K-duality and duality of extensions of $C^*$-algebras]{Spanier--Whitehead K-duality \\ and duality of extensions of $C^*$-algebras}
\author{Ulrich Pennig}
\address{Cardiff University, School of Mathematics, Senghennydd Road, Cardiff, CF24 4AG, Wales, United Kingdom}
\email{PennigU@cardiff.ac.uk}

\author{Taro Sogabe}
\address{Kyoto University, Department of Mathematics, Kitashirakawa Oiwake-cho, Sakyo-ku, Kyoto 606-8502, Japan}
\email{sogabe.taro.3v@kyoto-u.ac.jp}

\date{}
\begin{document}
\maketitle
\begin{abstract}
KK-theory is a bivariant and homotopy-invariant functor on $C^*$-algebras that combines K-theory and K-homology. KK-groups form the morphisms in a triangulated category. Spanier-Whitehead K-Duality intertwines the homological with the cohomological side of KK-theory. Any extension of a unital $C^*$-algebra by the compacts has two natural exact triangles associated to it (the extension sequence itself and a mapping cone sequence). We find a duality (based on Spanier-Whitehead K-duality) that interchanges the roles of these two triangles together with their six-term exact sequences. This allows us to give a categorical picture for the duality of Cuntz--Krieger--Toeplitz extensions discovered by K. Matsumoto. 
\end{abstract}

\section{Introduction}
Kasparov's KK-theory combines both of K-homology and K-theory of $C^*$-algebras into one additive bivariant functor. The KK-groups provide topological invariants that play a crucial role in index theory and the classification of nuclear $C^*$-algebras and their extensions. In this context they are used as a tool for understanding $*$-homomorphisms between them. In some situations the KK-groups even contain the complete information about the set of $*$-homomorphisms between two $C^*$-algebras up to homotopy: By the celebrated Kirchberg-Phillips theorem this happens, for example, for stable Kirchberg algebras (see \cite{RS} for a survey).
More precisely, for two stable Kirchberg algebras $A, B$, the theorem shows that every element of $KK(A, B)$ is represented by a $*$-homomorphism $A\to B$ and the choice of the $*$-homomorphism is unique up to homotopy. The composition of the homomorphisms is described by the Kasparov product. It is associative and unital and therefore allows us to view the KK-groups as morphisms in a category with rich additional structure, including several dualities. Moreover, it was shown in \cite{MN} that it is a tensor triangulated category.

In the present paper we analyse the interplay of Spanier--Whitehead K-duality with the triangulated structure. Classical Spanier--Whitehead duality takes place in the stable homotopy category of topological spaces. The dual of an object $X$ is witnessed by two morphisms $\nu_{X,Y} \colon X \wedge Y \to S^0$ and $\mu_{X,Y} \colon S^0 \to Y \wedge X$, which have to satisfy certain zig-zag relations. Following the idea that KK-theory can be viewed as stable homotopy theory for $C^*$-algebras this duality was transferred to the KK-category in \cite{KS, KP}, where the authors called it Spanier--Whitehead K-duality (see Def.~\ref{dsw}).

It was shown in \cite{KS} that any separable UCT $C^*$-algebra with finitely generated K-groups is dualizable with a separable UCT dual algebra.
For a separable $C^*$-algebra $A$ and its separable dual algebra $D(A)$,
the duality provides an isomorphism
\[KK(A, \mathbb{C})\cong KK(\mathbb{C}, D(A)).\]
Thus, the assumption that the K-groups are finitely generated cannot be dropped because separable algebras have countable K-groups.
The duality is defined by two duality classes
\[\mu_A\in KK(\mathbb{C}, A\otimes D(A)),\quad \nu_A\in KK(D(A)\otimes A, \mathbb{C})\]
which give rise to a group isomorphism
(see Sec. \ref{sd})
\[D_{\mu_A, \nu_B}(-) : KK(A, B)\to KK(D(B), D(A)).\]

The triangulated structure on the KK-category defined in \cite{MN} can be understood in terms of exact triangles induced by abstract mapping cone sequences
\[SA\xrightarrow{i(f)}C_f\xrightarrow{e(f)}B\xrightarrow{f}A\]
where the algebra $C_{f}$ denotes the mapping cone algebra of $f$.
This is a non-commutative analogue of the Puppe sequence in topology,
and we can extend this sequence as follows
\[\cdots SB\xrightarrow{Sf} SA\xrightarrow{i(f)}C_f\xrightarrow{e(f)}B\xrightarrow{f}A\xrightarrow{d(f)}SC_f\xrightarrow{Se(f)}SB\cdots\]
via a morphism $d(f)\in KK(A, SC_f)$ to obtain Puppe's exact sequence for the KK-groups (see Sec.~\ref{doe}, \cite[{Thm. 19.4.3.}]{B}).

With the duality and the tensor triangulated structure at hand it is a natural question whether and how these two are compatible. We give a partial answer to this question in 
our first main theorem of this paper:
\begin{thm}[{Thm. \ref{Mt1}, Cor. \ref{Mt1c}}]
Let $A, B, C$ be separable nuclear UCT $C^*$-algebras with dual algebras $D(A), D(B), D(C)$.
\begin{enumerate}[(1)]
\item 
Let $f : B\to A$ and $D(f) : D(A)\to D(B)$ be $*$-homomorphisms satisfying
\[D_{\mu_B, \nu_A}(KK(f))=KK(D(f))\in KK(D(A), D(B)),\]
then $SC_{D(f)}$ is a dual algebra of the mapping cone $C_f$ and there is a duality class $\mu\in KK(\mathbb{C}, C_f\otimes (SC_{D(f)}))$ satisfying
\[D_{\mu, \nu_B}(KK(e(f)))=d(D(f)).\]
\item For an appropriate choice of duality classes,
Spanier--Whitehead K-duality maps the exact triangle $SA\to C\to B\to A$ to another exact triangle $D(SA)\leftarrow D(C)\leftarrow D(B)\leftarrow D(A)$. 
\end{enumerate}
\end{thm}
The construction of the duality class $\mu \in KK(\mathbb{C}, C_f \otimes SC_{D(f)})$ is an adaptation to $C^*$-algebras of the one used in \cite[Lem.~14.31]{Sw}. It is based on a difference map 
\[
    \varphi \colon (SA \otimes D(A)) \oplus (SB \otimes D(B)) \to SA \otimes D(B),
\]
defined in Lem.~\ref{lem:diffmap}, whose mapping cone can be identified with a subalgebra of $C_f \otimes C_{D(f)}$. Since $KK(\varphi)$ kills the pair of (suspended) duality classes for $A$ and $B$, it gives rise to an element $s\mu \in KK(S, C_{\varphi})$, which gives $\mu$ via Bott periodicity and the appropriate identifications.

As an application of the above theorem,
we obtain a categorical picture 
of strong K-theoretic duality for extensions introduced in \cite{M} (see Def. \ref{skd}).
In \cite[{Thm. 1.1.}]{M},
K.~Matsumoto discovered an interesting duality between K-theory and strong extension groups of the Toeplitz extension of Cuntz--Krieger algebras (Rem. \ref{md}).
He introduces the notion of strong K-theoretic duality for unital extensions. Roughly speaking, two extensions $\mathbb{K} \to E \to A$ and $\mathbb{K} \to F \to B$ are dual to one another if the $Ext$-group six-term exact sequence of the first extension is isomorphic to the $K$-theory six-term exact sequence of the second and vice versa. As mentioned above, this duality relates the following Toeplitz extensions of Cuntz--Krieger algebras: 
\[
	\mathbb{K}\to \mathcal{T}_A\to\mathcal{O}_A,\quad \mathbb{K}\to\mathcal{T}_{A^t}\to\mathcal{O}_{A^t}\ .
\]
It was not known whether there are any other pairs of strongly K-theoretic dual extensions.
In  the second part of this paper,
we show the following theorem.
\begin{thm}[{Thm. \ref{mtcp}}]
Let $A$ be a unital separable nuclear UCT C*-algebras with finitely generated K-groups,
and let $\mathbb{K}\to E\to A$ be a unital essential extension.
Then, there exists a unital separable nuclear UCT C*-algebra $B$ and a unital essential extension $\mathbb{K}\to F\to B$ which is strongly K-theoretic dual to $\mathbb{K}\to E\to A$.
\end{thm}
We also show that strong K-theoretic duality can be understood as the Spanier--Whitehead duality of the following mapping cone sequences (Thm. \ref{mtcp} (2))
\[
\begin{tikzcd}
	C_{\xi_E} \ar{r}{e(\xi_E)} & \mathbb{K}+\mathbb{C}1_E \arrow[right hook->]{r}{\xi_E} & E, & F & \arrow[left hook->]{l}[above]{\xi_F} \mathbb{C}1_F+\mathbb{K} & \ar{l}[above]{e(\xi_F)} C_{\xi_F}\ .
\end{tikzcd}
\]
A key ingredient in the proof is the observation that all KK-theoretic information about the extension $\mathbb{K} \to E \to A$, i.e.\ its class in $Ext_s(A)$ and the $K$-theory class of the unit, is in fact encapsulated in the single KK-class $KK(e(\xi_E))$, where $e(\xi_E)$ is the evaluation map of the mapping cone $C_{\xi_E}$ (see Prop.~\ref{keyp2} proven in the appendix). To see this, we need the isomorphism  
\[
    \Psi_A \colon Ext_s(A) \to KK(C_{u_A}, \mathbb{C}),
\]
which is carefully defined in Cor.~\ref{dDd} in such a way that it maps the image of the generator of $K_1(\mathcal{Q}(K)) \cong \Z$ in $Ext_s(A)$ to the $KK$-class of the evaluation map on the cone. An interesting feature of strong K-theoretic duality revealed in this picture is that it interchanges the roles of $\mathbb{C}$ and $\mathbb{K}$ under the duality. The dual extension can then be constructed by first replacing $E$ by a KK-equivalent Kirchberg algebra $R$ and then using reciprocality defined in \cite{S}. In fact, the mapping cone $C_{u_B}$ of the reciprocal algebra $B$ is a dual of $R \sim_{KK} E$. Using the inverse of $\Psi_B$ it is then easy to check, for example, that $B$ has the correct extension groups.

The paper is structured as follows: We fix some notation used throughout the paper in Section~2.

We then recall some basic facts about KK-groups and exact triangles in KK at the beginning of Section~3. We continue with a discussion of Spanier--Whitehead K-Duality (see Def.~\ref{dsw}) and its properties. The main focus of Sec.~3 is the construction of the difference map $\varphi$ in Lem.~\ref{lem:diffmap}, the identification of its mapping cone $C_{\varphi}$ (Lem.~\ref{K}), the construction of the duality classes in Lem.~\ref{cii} and finally the proof of Lem.~\ref{Mt1} and Cor.~\ref{Mt1c}.

In Section~4 we first state some well-known theorems about (strong and weak) extension groups, Busby invariants and the six-term exact sequence of extension groups. We also recall the identification of the strong extensions of $A$ with $Ext(C_{u_A}, S\mathbb{K})$. In the rest of this section we then construct the isomorphism $\Psi_A$ mentioned above in Cor.~\ref{dDd}.

The second main result is discussed in Section~5. Some of the proofs are deferred to the appendix in Section~6. We recall  the definition of (strong) K-theoretic duality for unital extensions (see Def.~\ref{skd}) at the beginning. As mentioned above, Prop.~\ref{keyp2} and Prop.~\ref{keyp1} provide an interpretation of the duality in terms of cones, which is then used in the proof of the main result, Thm.~\ref{mtcp}. The construction of the dual extension using reciprocality can be found in Lem.~\ref{exB}.

\section*{Acknowledgements}
This work was supported by the Research Institute for Mathematical Sciences,
an International Joint Usage/Research Center located in Kyoto University.
T. Sogabe is supported by Research Fellow of the Japan Society for the Promotion of Science.

\section{Notation}
Throughout the paper the letters $A, B, C$ denote separable nuclear UCT C*-algebras.
If $A$ is unital, then we denote by $1_A\in A$ its unit and write 
$u_A : \mathbb{C}\ni \lambda\mapsto \lambda 1_A\in A.$
Let $\mathbb{K}$ (resp.~$\mathbb{M}_n$) denote the algebra of compact operators acting on a separable infinite dimensional (resp. $n$-dimensional) Hilbert space,
and let $e\in \mathbb{K}$ denote a rank 1 projection.
We also write $e : \mathbb{C}\ni\lambda\mapsto \lambda e\in\mathbb{K}$ for the corresponding $*$-homomorphism by abuse of notation.
We denote by $\mathcal{M}(B\otimes\mathbb{K})$ the multiplier algebra of $B\otimes\mathbb{K}$ and write $\mathcal{Q}(B\otimes\mathbb{K}):=\mathcal{M}(B\otimes\mathbb{K})/(B\otimes\mathbb{K})$.
The quotient map $\mathcal{M}(B\otimes\mathbb{K})\to\mathcal{Q}(B\otimes\mathbb{K})$ is denoted by $\pi$ for short.
We denote by
\[\sigma_{A, B} : A\otimes B\ni a\otimes b\mapsto b\otimes a\in  B\otimes A\]
the flip isomorphism.

We denote by $K_i(A), i=0, 1$ the $i$-th K-group of $A$ and write $[p]_0\in K_0(A)$ (resp. $[u]_1\in K_1(A)$) the equivalence class of the projection $p$ (resp. unitary $u$).
We identify $K_0(\mathbb{K})$ and $K_1(\mathcal{Q}(\mathbb{K}))$ with $\mathbb{Z}$ as follows
\begin{align*}
&K_0(\mathbb{K})\ni [e]_0\mapsto 1\in \mathbb{Z},\\
&K_1(\mathcal{Q}(\mathbb{K}))\ni [\pi (V)]_1\mapsto {\rm Index} (V)\in\mathbb{Z},
\end{align*}
where $V\in \mathcal{M}(\mathbb{K})$ is a Fredholm operator and we write ${\rm Index} (V):={\rm dim\; Ker}V-{\rm dim\; Coker}V.$
We denote the index map by
\[{\rm Ind} : K_1(\mathcal{Q}(\mathbb{K}))\ni [\pi(V)]_1\mapsto {\rm Index} (V)\in K_0(\mathbb{K}).\]

Let $S=C_0(0, 1)$ be the algebra of continuous functions on $[0, 1]$ vanishing at the boundary.
For a $*$-homomorphism $f : B\to A$,
we write $S^nA:=S^{\otimes n}\otimes A, \; S^nf:={\rm id}_{S^n}\otimes f$.
A function $[0, 1]\ni t\mapsto a(t)\in A$ vanishing at $\{0, 1\}$ is an element of $SA$ and denoted by $a(t)\in SA$ by abuse of notation.
The mapping cone algebra $C_f$ of $f$ is defined by
\[C_f:=\{(a(t), b)\in (C_0(0, 1]\otimes A)\oplus B\; |\; a(1)=f(b)\},\]
and it fits into an exact sequence
\[0\to SA\xrightarrow{i(f)}C_f\xrightarrow{e(f)} B\to 0,\]
where the two maps $i(f)$ and $e(f)$ are given by
\[ i(f) : SA\ni a(t)\mapsto (a(t), 0)\in C_f,\quad e(f) : C_f\ni (a(t), b)\mapsto b\in B.\] 
For an extension
\[0\to J\to E\xrightarrow{f} E/J\to 0,\]
we write
$j(E) : J\ni x\mapsto (0, x)\in C_f.$

\section{Exactness of Spanier--Whitehead K-duality}
In this section, we will show that, for an appropriate choice of duality classes, Spanier-Whitehead K-duality maps an exact triangle to another exact triangle (see Thm. \ref{Mt1}, Cor. \ref{Mt1c}).

\subsection{KK-groups and exact triangles}
We refer to \cite{B} for the basic definition and facts about KK-theory.
For two C*-algebras $A, B$,
Kasparov's KK-group is denoted by $KK(A, B)$.
The Kasparov module given by a $*$-homomorphism $f : A\to B$ is denoted by $KK(f)\in KK(A, B)$, and for two $*$-homomorhisms $f : A\to B$, $g : B\to C$, their Kasparov product is denoted by
\[\hat{\otimes} : KK(A, B)\times KK(B, C)\ni (KK(f), KK(g))\mapsto KK(f)\hat{\otimes}KK(g)=KK(g\circ f)\in KK(A, C).\]
We write $I_A:=KK({\rm id}_A)\in KK(A, A)$, and one has natural maps
\begin{align*}
I_A\otimes - &: KK(B, C)\ni KK(g)\mapsto I_A\otimes KK(g)=KK({\rm id}_A\otimes g)\in KK(A\otimes B, A\otimes C), \\
-\otimes I_A &: KK(B, C)\ni KK(g)\mapsto KK(g)\otimes I_A=KK(g\otimes {\rm id}_A)\in KK(B\otimes A, C\otimes A).	
\end{align*}
The following identities are consequences of the definition of Kasparov modules:
\begin{enumerate}
\item $I_A\otimes b =(b\otimes I_A)\hat{\otimes} KK(\sigma_{B, A})\in KK(A, A\otimes B),\quad b\in KK(\mathbb{C}, B)$,
\item $I_A\otimes c=KK(\sigma_{A, C})\hat{\otimes}(c\otimes I_A)\in KK(A\otimes C, A),\quad c\in KK(C, \mathbb{C})$,
\item $(a\otimes I_C)\hat{\otimes}(I_B\otimes c)=(I_A\otimes c)\hat{\otimes}(a\otimes I_D),\quad a\in KK(A, B),\; c\in KK(C, D)$,
\end{enumerate}
and will be used in this paper without mentioning.

We denote by $KK$ the category of separable C*-algebras whose morphism set $Mor(A, B)$ is $KK(A, B)$,
and the composition of the morphisms is given by $\hat{\otimes}$.
A morphism $\alpha \in KK(A, B)$ is called a KK-equivalence if there exists $\alpha^{-1}\in KK(B, A)$ satisfying $\alpha\hat{\otimes}\alpha^{-1}=I_A,\; \alpha^{-1}\hat{\otimes}\alpha=I_B$ and we denote by $KK(A, B)^{-1}$ the subset of KK-equivalences.
If $KK(A, B)^{-1}\not=\emptyset$, $A$ and $B$ are called KK-equivalent.

In \cite{MN},
R. Meyer and R. Nest showed that  $KK$ is triangulated.
A sequence 
\[S\tilde{A}\to \tilde{C}\to \tilde{B}\to \tilde{A}\]
in $KK$ is called an exact triangle if there is a $*$-homomorphism $f : B\to A$ and $KK$-equivalences
$\alpha \in KK(\tilde{A}, A)^{-1}$, $\beta\in KK(\tilde{B}, B)^{-1}$, $\gamma\in KK(\tilde{C}, C_f)^{-1}$ making the following diagram commute:

\[
\begin{tikzcd}
	S\tilde{A}\ar[r]\ar[d,"I_S\otimes\alpha"] & \tilde{C}\ar[r]\ar[d,"\gamma"] &\tilde{B}\ar[d,"\beta"]\ar[r]&\tilde{A}\ar[d,"\alpha"]\\
SA\ar[r,"i(f)" below]&C_f\ar[r,"e(f)" below] &B\ar[r,"f" below] &A.
\end{tikzcd}
\]
\begin{lem}[{\cite[{Sec. 2.1, Appendix}]{MN}}]\label{mn}
For two exact triangles $SA_i\to C_i\to B_i\to A_i, \;i=1, 2$ fitting into the commutative diagram in $KK$ shown below
\[
\begin{tikzcd}
SA_1\ar[r]\ar[d,"I_S\otimes \alpha"]&C_1\ar[r]&B_1\ar[r]\ar[d,"\beta"] &A_1\ar[d,"\alpha"]\\
SA_2\ar[r]&C_2\ar[r]&B_2\ar[r]&A_2,
\end{tikzcd}
\]
we have $\gamma\in KK(C_1, C_2)$ making the above diagram commute.
If $\alpha, \beta$ are KK-equivalences, then so is $\gamma$.
\end{lem}

\subsection{Spanier--Whitehead K-duality}\label{sd}
In this section we recall the definition of Spanier--Whitehead K-duality and the existence of duals for C*-algebras with finitely generated K-groups following \cite{KS}.
\begin{dfn}[{\cite[Def. 2.1]{KS}}]\label{dsw}
Let $A$ and $D(A)$ be separable C*-algebras.
They are Spanier--Whitehead K-dual if and only if there are elements, called duality classes,
\[\mu_A \in KK(\mathbb{C}, A\otimes D(A)), \quad \nu_A \in KK(D(A)\otimes A, \mathbb{C})\]
satisfying the unit co-unit adjunction formula
\[(\mu_A \otimes I_A)\hat{\otimes} (I_A\otimes \nu_A)=I_A\in KK(A, A),\quad (I_{D(A)}\otimes\mu_A)\hat{\otimes}(\nu_A\otimes I_{D(A)})=I_{D(A)}\in KK(D(A), D(A)).\]
\end{dfn}
\begin{rem}
We frequently denote by $D(A)$ a dual algebra of $A$. It should be noted, however, that this notation is misleading. There is in general no way to determine $D(A)$ from $A$ up to isomorphism, only up to KK-equivalence.
Let $A$ and $D$ be dual with the duality classes $\mu\in KK(\mathbb{C}, A\otimes D)$ and $\nu\in KK(D\otimes A, \mathbb{C})$.
If there is another algebra $\tilde{D}$ with a KK-equivalence $\gamma \in KK(D, \tilde{D})^{-1}$, $\tilde{D}$ is also a dual of $A$ with the following duality classes
\[(\mu\hat{\otimes}(I_A\otimes\gamma),  (\gamma^{-1}\otimes I_A)\hat{\otimes}\nu).\]
Duality classes $(\mu, \nu)$ are determined up to $KK(A, A)^{-1}\cong KK(D(A), D(A))^{-1}$ as in the above manner (see for example \cite[Lem. 2.10]{S}).
If we fix one of the duality classes then the other is uniquely determined.
So if $A$ and $D$ are dual and $(\mu, \nu_i),\; i=1, 2$ are duality classes, then one has $\nu_1=\nu_2\in KK(D\otimes A, \mathbb{C})$.
\end{rem}
We have the following characterization of the duality.
\begin{lem}\label{ci}
Let $A, D, P, Q$ be separable nuclear C*-algebras, and let $\mu \in KK(\mathbb{C}, A\otimes D)$ be an element such that the natural transformation
\[\mu\hat{\otimes} : KK(P\otimes A, Q)\ni x\mapsto (I_P\otimes \mu)\hat{\otimes}(x\otimes I_D)\in KK(P, Q\otimes D)\]
is an isomorphism for any $P, Q$.
Then, there exists $\nu\in KK(D\otimes A, \mathbb{C})$,
and $A$ and $D$ are the Spanier--Whitehead K-dual with duality classes $(\mu, \nu)$.
\end{lem}
\begin{proof}
Since $\mu\hat{\otimes} : KK(D\otimes A, \mathbb{C})\cong KK(D, D)$,
the element $\nu:=(\mu\hat{\otimes})^{-1}(I_D)$ satisfies
\[(I_D\otimes \mu)\hat{\otimes}(\nu\otimes I_D)=I_D,\]
and the above equation implies
\[\mu\hat{\otimes} : KK(A, A)\ni (\mu\otimes I_A)\hat{\otimes}(I_A\otimes \nu)\mapsto \mu\in KK(\mathbb{C}, A\otimes D).\]
We also have
\[\mu\hat{\otimes} : KK(A, A)\ni I_A\mapsto \mu\in KK(\mathbb{C}, A\otimes D),\]
and the assumption shows $(\mu\otimes I_A)\hat{\otimes}(I_A\otimes \nu)=I_A$.
\end{proof}
We have an easy picture to understand the duality if we focus on the UCT C*-algebras with finitely generated K-groups as in the following theorem.
\begin{thm}[{c.f. \cite[{Thm. 3.1, 6.2}]{KS}}]\label{ks}
Let $A$ be a separable nuclear UCT C*-algebra with finitely generated K-groups.
Then, $A$ is KK-equivalent to an algebra of the following form
\[\mathbb{C}^{\oplus a}\oplus S^{\oplus b}\oplus \mathcal{O}_{n+1}^{\oplus c}\oplus (S\mathcal{O}_{m+1})^{\oplus d}\oplus\cdots\]
and a dual algebra $D(A)$ is given by
\[D(A):=\mathbb{C}^{\oplus a}\oplus S^{\oplus b}\oplus (S\mathcal{O}_{n+1})^{\oplus c}\oplus \mathcal{O}_{m+1}^{\oplus d}\oplus\cdots.\]
\end{thm}
\begin{rem}
The C*-algebra $\mathcal{O}_{n+1}$ is the Cuntz algebra whose K-groups are given by
\[K_0(\mathcal{O}_{n+1})=\mathbb{Z}/n\mathbb{Z},\quad K_1(\mathcal{O}_{n+1})=0.\]
Two UCT C*-algebras are KK-equivalent if and only if they have isomorphic K-groups,
and the finitely generated K-groups of $A$
\[K_0(A)=\mathbb{Z}^{\oplus a}\oplus (\mathbb{Z}/n\mathbb{Z})^{\oplus c}\oplus\cdots, \quad K_1(A)=\mathbb{Z}^{\oplus b}\oplus (\mathbb{Z}/m\mathbb{Z})^{\oplus d}\oplus\cdots\]
are the same as the K-groups of the UCT C*-algebra $\mathbb{C}^{\oplus a}\oplus S^{\oplus b}\oplus \mathcal{O}_{n+1}^{\oplus c}\oplus (S\mathcal{O}_{m+1})^{\oplus d}\oplus\cdots$.
\end{rem}
\begin{rem}
In \cite{KP},
two Cuntz--Krieger algebras $\mathcal{O}_A$ and $\mathcal{O}_{A^t}$ are proved to be dual with respect to duality classes which are elements in the respective $KK^1$-groups.
This implies $D(\mathcal{O}_A)=S\mathcal{O}_{A^t}$ in our setting.
In particular, one has $D(\mathcal{O}_{n+1})=S\mathcal{O}_{n+1}$.
\end{rem}
Obviously, $\mathbb{C}$ is self-dual with $\mu =\nu=\pm  I_\mathbb{C}$.
Since $KK(\sigma_{S, S})=-I_{S^{2}}$,
it is also easy to check that $S$ is self-dual with the duality classes $(\beta_S, \beta_S^{-1})$ given by the Bott element $\beta_S \in KK(\mathbb{C}, S^{2})$.
Let $\mathbb{K}+\mathbb{C}1$ denote the unitization of $\mathbb{K}$ in $\mathcal{M}(\mathbb{K})$ (i.e., $1=1_{\mathcal{M}(\mathbb{K})}$).
Using the UCT
\[KK(\mathbb{C}, (\mathbb{K}+\mathbb{C}1)^{\otimes 2})=\operatorname{Hom}(K_0(\mathbb{C}), K_0((\mathbb{K}+\mathbb{C}1)^{\otimes 2})),\]
\[KK((\mathbb{K}+\mathbb{C}1)^{\otimes 2}, \mathbb{C})=\operatorname{Hom}(K_0((\mathbb{K}+\mathbb{C}1)^{\otimes 2}), K_0(\mathbb{C})),\]
we define two elements $\mu_{\epsilon, \delta}\in KK(\mathbb{C}, (\mathbb{K}+\mathbb{C}1)^{\otimes 2})$ and $\nu_{\epsilon, \delta}\in KK((\mathbb{K}+\mathbb{C}1)^{\otimes 2}, \mathbb{C})$ by
\begin{align*}
&\mu_{\epsilon, \delta} : [1_\mathbb{C}]_0\mapsto \epsilon [e\otimes 1]_0+\delta [1\otimes e]_0,\\
&\nu_{\epsilon, \delta} : [1\otimes e]_0\mapsto \epsilon [1_\mathbb{C}]_0,\\
&\nu_{\epsilon, \delta} : [e\otimes 1]_0\mapsto \delta [1_\mathbb{C}]_0,\\
&\nu_{\epsilon, \delta} : [e\otimes e]_0\mapsto 0,\\
&\nu_{\epsilon, \delta} : [1\otimes 1]_0\mapsto 0
\end{align*}
for $\epsilon, \delta \in \{\pm 1\}$.
One can easily check the following lemma.
\begin{lem}\label{dc}
The algebra $\mathbb{K}+\mathbb{C}1$  is self-dual with the duality classes $(\mu_{\epsilon,\delta}, \nu_{\epsilon, \delta})$. 
\end{lem}
Let $A$ and $B$ be dualizable C*-algebras with dual algebras $D(A), D(B)$,
and let $(\mu_A, \nu_A), (\mu_B, \nu_B)$ be their duality classes.
One has two isomorphisms
\[KK(A, B)\ni x\mapsto \mu_A\hat{\otimes}(x\otimes I_{D(A)})\in KK(\mathbb{C}, B\otimes D(A)),\]
\[KK(\mathbb{C}, B\otimes D(A))\ni y\mapsto (I_{D(B)}\otimes y)\hat{\otimes}(\nu_B\otimes I_{D(A)})\in KK(D(B), D(A)),\]
where the inverse of the first map is
\[KK(\mathbb{C}, B\otimes D(A))\ni y\mapsto (y\otimes I_A)\hat{\otimes}(I_B\otimes \nu_A)\in KK(A, B)\]
and the inverse of the second one is given similarly.
Thus, we have the following isomorphism
\[D_{\mu_A, \nu_B}(-) : KK(A, B)\ni x\mapsto (I_{D(B)}\otimes\mu_A)\hat{\otimes}(I_{D(B)}\otimes x\otimes I_{D(A)})\hat{\otimes}(\nu_B\otimes I_{D(A)})\in KK(D(B), D(A))\]
which provides a dual morphism $D(A)\xleftarrow{D_{\mu_A, \nu_B}(x)}D(B)$ for a given morphism $A\xrightarrow{x} B$.
Using the three equations for the Kasparov product listed in the previous section,
a direct computation yields the following lemma.
\begin{lem}\label{inv}
Let $(\mu_A, \nu_A)$ be duality classes for $A$ and $D(A)$.
Then, the following elements 
$$\mu_A\hat{\otimes}KK(\sigma_{A, D(A)})\in KK(\mathbb{C}, D(A)\otimes A),\quad KK(\sigma_{A, D(A)})\hat{\otimes}\nu_A\in KK(A\otimes D(A),\mathbb{C})$$ are duality classes for $D(A)$ and $A$,
and one has $$D_{\mu_A, \nu_B}(-)^{-1}=D_{\mu_B\hat{\otimes}\sigma_{B, D(B)}, \sigma_{A, D(A)}\hat{\otimes}\nu_A}(-).$$
\end{lem}



\subsection{Duals of exact triangles}\label{doe}
For $SA\xrightarrow{i(f)}C_f\xrightarrow{e(f)}B\xrightarrow{f}A$ and the Bott element $\beta_S\in KK(\mathbb{C}, S^{\otimes 2})$,
one has a morphism
\[d(f):=(\beta_S\otimes I_A)\hat{\otimes}(I_S\otimes KK(i(f)))\in KK(A, SC_f).\]
In the following,
we fix duality classes $(\mu_A, \nu_A)$, $(\mu_B, \nu_B)$ and dual algebras $D(A), D(B)$ as in Def.~\ref{dsw}.
Assume that there is a $*$-homomorphism $D(f) : D(A)\to D(B)$ satisfying
\[D_{\mu_B, \nu_A}(KK(f))=KK(D(f))\in KK(D(A), D(B)).\]
Note that one can always chose $D(A), D(B)$ as stable Kirchberg algebras,
and then $D_{\mu_B, \nu_A}(KK(f))$ is represented by a $*$-homomorphism $D(f) : D(A)\to D(B)$. 

Following \cite[Chap. 14]{Sw},
we will show the following theorem in the next two subsections.
\begin{thm}\label{Mt1}
Let $f \colon B \to A$ be a $*$-homomorphism with mapping cone algebra $C_f$ and dual homomorphism $D(f) \colon D(A) \to D(B)$ as described above. The two algebras $C_f$ and $SC_{D(f)}$ are Spanier--Whitehead K-dual with a duality class
\[\mu\in KK(\mathbb{C}, C_f\otimes SC_{D(f)})\]
satisfying
\begin{align*}
D_{\mu, \nu_B}(KK(e(f)))=&(I_{D(B)}\otimes\mu)\hat{\otimes}(I_{D(B)}\otimes KK(e(f))\otimes I_{SC_{D(f)}})\hat{\otimes}(\nu_B\otimes I_{SC_{D(f)}})\\
=&d(D(f)).
\end{align*}
\end{thm}

\subsubsection{Construction of $\mu$}
Let $\phi : S\oplus S\to S$ be the map sending $(\alpha, \beta)\in S\oplus S$ to the function $\gamma \in S$ defined by
\[\gamma(t):=\alpha(2t),\; t\in [0, 1/2], \quad \gamma(t):=\beta(2-2t),\; t\in [1/2, 1].\]
We identify $KK(A, (S\oplus S)\otimes B)$ with $KK(A, SB)^{\oplus 2}$,
and $\phi$ induces a map
\[-\hat{\otimes}(KK(\phi)\otimes I_B) : KK(A, (S\oplus S)\otimes B)\ni (\alpha, \beta)\mapsto \alpha-\beta \in KK(A, SB).\]
\begin{lem} \label{lem:diffmap}
Let $\varphi : (SA\otimes D(A))\oplus (SB\otimes D(B))\to SA\otimes D(B)$ be the $*$-homomorphism
\[\varphi : = (\phi\otimes {\rm id}_{A\otimes D(B)})\circ(({\rm id}_{SA}\otimes D(f))\oplus({\rm id}_S\otimes f\otimes {\rm id}_{D(B)})).\]
For $(I_S\otimes \mu_A, I_S\otimes \mu_B)\in KK(S, (SA\otimes D(A))\oplus (SB\otimes D(B)))$,
we have \[(I_S\otimes \mu_A, I_S\otimes \mu_B)\hat{\otimes} KK(\varphi)=0\in KK(S, SA\otimes D(B)).\]
\end{lem}
\begin{proof}
By the definition of $D(f)$,
one has
\begin{align*}
\mu_B\hat{\otimes}KK(f\otimes {\rm id}_{D(B)})&=(I_\mathbb{C}\otimes\mu_B)\hat{\otimes}(I_\mathbb{C}\otimes KK(f)\otimes I_{D(B)})\\
&=(I_\mathbb{C}\otimes \mu_B)\hat{\otimes}(I_\mathbb{C}\otimes KK(f)\otimes I_{D(B)})\hat{\otimes}(\mu_A\otimes I_{A\otimes D(B)})\hat{\otimes}(I_A\otimes \nu_A\otimes I_{D(B)})\\
&=\mu_A\hat{\otimes}(I_A\otimes KK(D(f))).
\end{align*}
Thus, we obtain
\begin{align*}
&(I_S\otimes \mu_A, I_S\otimes \mu_B)\hat{\otimes}KK(\varphi)\\
=&(I_S\otimes(\mu_A\hat{\otimes}(I_A\otimes KK(D(f)))), I_S\otimes(\mu_B\hat{\otimes}(KK(f)\otimes I_{D(B)})))\hat{\otimes}(KK(\phi)\otimes I_{A\otimes D(B)})\\
=&I_S\otimes(\mu_A\hat{\otimes}(I_A\otimes KK(D(f)))-\mu_B\hat{\otimes}(KK(f)\otimes I_D(B)))\\
=&0. \qedhere
\end{align*}
\end{proof}
Combining the above lemma with Puppe's exact sequence (see \cite[{Thm. 19.4.3.}]{B})
\[KK(S, C_\varphi)\xrightarrow{-\hat{\otimes}KK(e(\varphi))}KK(S, S(A\otimes D(A))\oplus S(B\otimes D(B)))\xrightarrow{-\hat{\otimes}KK(\varphi)} KK(S, SA\otimes D(B)),\]
there is an element $s\mu \in KK(S, C_\varphi)$ satisfying $s\mu\hat{\otimes}KK(e(\varphi))=(I_S\otimes \mu_A, I_S\otimes \mu_B)$.

Below we will identify the mapping cone $C_\varphi$ with
\[K_{f, D(f)}:=\operatorname{Ker}(C_f\otimes C_{D(f)}\xrightarrow{e(f)\otimes e(D(f))}B\otimes D(A)).\]
By the definition of the mapping cone algebras,
it is easy to check that $C_\varphi$ is identified with a subalgebra of
\[C([0, 1]^2, A\otimes D(B))\oplus C([0, 1], A\otimes D(A))\oplus C([0, 1], B\otimes D(B))\]
where an element $(F(-, -), a(-), b(-))$ lies in $C_\varphi$ if and only if it satisfies the following conditions:
\begin{align*}
& F(p, 0)=F(0, q)=F(p, 1)=0, a(0)=a(1)=0,\; b(0)=b(1)=0,\\
&F(1, q)={\rm id}_A\otimes D(f)(a(2q)),\; q\in [0, 1/2],\\
&F(1, q)=f\otimes {\rm id}_{D(B)}(b(2-2q)),\; q\in [1/2, 1].
\end{align*}
Recall that $C_f\otimes C_{D(f)}$ is a subalgebra of
\[((C_0(0, 1]\otimes A)\oplus B)\otimes ((C_0(0, 1]\otimes D(B))\oplus D(A)).\]
\begin{lem}\label{K}
The algebra $K_{f, D(f)}$ is identified with a subalgebra of 
\[C([0, 1]^2, A\otimes D(B))\oplus C([0, 1], A\otimes D(A))\oplus C([0, 1], B\otimes D(B))\]
consisting of the functions satisfying the following boundary conditions:
\begin{align*}
&F(0, s)=F(t, 0)=0,\; a(0)=a(1)=0,\; b(0)=b(1)=0,\\
&F(t, 1)={\rm id}\otimes D(f)(a(t)),\; F(1, s)=f\otimes {\rm id}(b(s)).
\end{align*}
\end{lem}
\begin{proof}
We write $CA:=C_0(0, 1]\otimes A$ for short and define a completely bounded map by
\[ Ev_1-f : CA\oplus B\ni (\alpha(t), \beta)\mapsto \alpha(1)-f(\beta)\in A.\]
By \cite[{Page 12}]{Pis},
one has the diagram below whose vertical and horizontal sequences are exact:
\[
\begin{tikzcd}[column sep=1.2cm, font=\small]
&& 0\ar[d] && 0\ar[d] \\
0 \ar[r] & C_f\otimes C_{D(f)} \ar[r] & C_f\otimes (CD(B)\oplus D(A)) \ar[rr,"{\rm id} \otimes (Ev_1-D(f))"] \ar[d] && C_f\otimes D(B) \ar[d] \\
&& (CA\oplus B)\otimes (CD(B)\oplus D(A)) \ar[d,"(Ev_1-f) \otimes {\rm id}"] \ar[rr,"{\rm id}\otimes (Ev_1-D(f))"] && (CA\oplus B)\otimes D(B)\ar[d,"(Ev_1-f)\otimes {\rm id}"]\\
&& A\otimes (CD(B)\oplus D(A)) \ar[rr,"{\rm id}\otimes (Ev_1-D(f))"] && A\otimes D(B).
\end{tikzcd}
\]
The above diagram implies
\[C_f\otimes C_{D(f)}={\rm Ker}((Ev_1-f)\otimes {\rm id})\cap {\rm Ker} ({\rm id}\otimes (Ev_1-D(f))).\]
Identifying
\[(\alpha(t), \beta)\otimes (x(s), y)\in (CA\oplus B)\otimes (CD(B)\oplus D(A))\]
with
\begin{align*}
&(\alpha(t)\otimes x(s), \alpha(t)\otimes y, \beta\otimes x(s), \beta\otimes y)=(F(t, s), a(t), b(s), d)\\
&\in C_0((0, 1]^2, A\otimes D(B))\oplus C_0((0, 1], A\otimes D(A))\oplus C_0((0, 1], B\otimes D(B))\oplus (B\otimes D(A)),
\end{align*}
one has
\begin{align*}
&{\rm id}\otimes (Ev_1-D(f)) (F, a, b, d)=(F(t, 1)-{\rm id}\otimes D(f)(a(t)), b(1)-{\rm id}\otimes D(f)(d)),\\
&(Ev_1-f)\otimes{\rm id} (F, a, b, d)=(F(1, s)-f\otimes{\rm id}(b(s)), a(1)-f\otimes{\rm id}(d)),\\
&K_{f, D(f)}=C_f\otimes C_{D(f)}\cap C_0((0, 1]^2, A\otimes D(B))\oplus C_0((0, 1], A\otimes D(A))\oplus C_0((0, 1], B\otimes D(B))\oplus 0.
\end{align*}
Now it is straightforward to prove the statement.
\end{proof}
We define a map $r : [0, 1]^2\to [0, 1]^2$ by
\[r (p, q):=(2qp, p), \; q\in [0, 1/2],\quad r (p, q):=(p, (2-2q)p), \; q\in [1/2, 1],\]
(see \cite[proof of Lem.~14.30]{Sw} for a sketch) and this map induces an isomorphism
\[r^* : K_{f, D(f)}\ni (F(t, s), a, b)\mapsto (F(r(p, q)), a, b)\in C_\varphi.\]
Let  $\mu \in KK(\mathbb{C}, C_f\otimes (SC_{D(f)}))$ be the composition of the following morphisms:
\[\mathbb{C}\xrightarrow{\beta_S}S^2\xrightarrow{I_S\otimes(s\mu)}SC_\varphi\xrightarrow{KK(S({r^*}^{-1}))}SK_{f, D(f)}\subset S(C_f\otimes C_{D(f)})\xrightarrow{KK(\sigma_{S, C_f}\otimes {\rm id}_{C_{D(f)}})} C_f\otimes (SC_{D(f)}).\]

\subsubsection{Proof of Thm. \ref{Mt1}}
\begin{lem}\label{sl}
Let $(\mu_A, \nu_A)$ be the duality classes for $A, D(A)$.
Then the following elements are duality classes for $SA, SD(A)$:
\[\mu_{SA}:=\mu_A\hat{\otimes}(\beta_S\otimes I_{A\otimes D(A)})\hat{\otimes}(KK(\sigma_{S, SA})\otimes I_{D(A)})\in KK(\mathbb{C}, SA\otimes SD(A)),\]
\[\nu_{SA}:=(KK(\sigma_{SD(A), S})\otimes I_A)\hat{\otimes}(\beta_S^{-1}\otimes I_{D(A)\otimes A})\hat{\otimes}\nu_A\in KK(SD(A)\otimes SA, \mathbb{C}).\]
\end{lem}
\begin{proof}
We check the unit co-unit adjunction formula.
One has
\[\mu_{SA}\otimes I_S=(I_S\otimes \mu_{SA})\hat{\otimes}KK(\sigma_{S, SA\otimes SD(A)}),\]
\[I_A\otimes \beta_S^{-1}=KK(\sigma_{A, S^2})\hat{\otimes}(\beta_S^{-1}\otimes I_A).\]
Thus, a direct computation yields
\begin{align*}
&\mu_{SA}\otimes I_{SA}\\
=&(I_S\otimes \mu_A\otimes I_A)\hat{\otimes}(I_S\otimes\beta_S\otimes I_{A\otimes D(A)\otimes A})\hat{\otimes}(KK(\sigma_{S, S^2A\otimes D(A)})\otimes I_A)\hat{\otimes}(KK(\sigma_{S, SA})\otimes I_{D(A)\otimes SA}), \\
&I_{SA}\otimes \nu_{SA} \\
=&(I_S\otimes KK(\sigma_{A\otimes SD(A), S})\otimes I_A)\hat{\otimes}(I_{S^2}\otimes KK(\sigma_{A, S})\otimes I_{D(A)\otimes A})\hat{\otimes}(I_S\otimes \beta_S^{-1}\otimes I_{A\otimes D(A)\otimes A})\hat{\otimes}(I_{SA}\otimes\nu_A).  
\end{align*}
In symmetric tensor categories we can represent composition of morphisms as string diagrams, in which the duality classes correspond to cups and caps and the transposition of tensor factors as crossings. The graphical representation of the above identity is shown in Fig.~\ref{fig:muSA}.

\begin{figure}[htp]
\begin{center}
\begin{tikzpicture}[scale=0.7,every node/.style={scale=0.7}]
	\node[draw=red,fill=white] at (0,0.3) {$S$};
	\node[draw=black,fill=white] at (5,0.3) {$A$};
	
	\draw[red,thick] (0,0) -- (0,-2);
	\draw (5,0) -- (5,-4);
	\draw plot [smooth] coordinates {(3,-1) (3.07,-.7) (3.2,-0.46) (3.5,-0.3) (3.8,-0.46) (3.93,-.7) (4,-1)};
	\draw plot [smooth] coordinates {(1,-2) (1.07,-1.7) (1.2,-1.46) (1.5,-1.3) (1.8,-1.46) (1.93,-1.7) (2,-2)};
	\draw (3,-1) -- (3,-2);
	\draw (4,-1) -- (4,-2);
	\draw[red,thick] plot [smooth,tension=.4] coordinates {(0,-2) (0.5,-2.4) (3.5,-2.5) (4,-3)};
	\draw plot [smooth] coordinates {(1,-2) (0.9,-2.3) (0.1,-2.7) (0,-3)};
	\draw plot [smooth] coordinates {(2,-2) (1.9,-2.3) (1.1,-2.7) (1,-3)};
	\draw plot [smooth] coordinates {(3,-2) (2.9,-2.3) (2.1,-2.7) (2,-3)};
	\draw plot [smooth] coordinates {(4,-2) (3.9,-2.3) (3.1,-2.7) (3,-3)};

	\draw plot [smooth] coordinates {(0,-3) (0.2,-3.3) (1.8,-3.6) (2,-4)};
	\draw plot [smooth] coordinates {(1,-3) (0.9,-3.3) (0.1,-3.7) (0,-4)};
	\draw plot [smooth] coordinates {(2,-3) (1.9,-3.3) (1.1,-3.7) (1,-4)};
	\draw (3,-3) -- (3,-4);
	\draw[red, thick] (4,-3) -- (4,-4);

	\node[draw=black,fill=white] at (3,-1) {$A$};
	\node[draw=black,fill=white,inner sep=2pt] at (4,-1) {$D(A)$};
	\node[draw=black,fill=white] at (1,-1.9) {$S$};
	\node[draw=black,fill=white] at (2,-1.9) {$S$};

	\node[draw=black,fill=white] at (0,-4.3) {$S$};
	\node[draw=black,fill=white] at (1,-4.3) {$A$};
	\node[draw=black,fill=white] at (2,-4.3) {$S$};
	\node[draw=black,fill=white,inner sep=2pt] at (3,-4.3) {$D(A)$};
	\node[draw=red,fill=white] at (4,-4.3) {$S$};
	\node[draw=black,fill=white] at (5,-4.3) {$A$};

	\node at (7,-2.5) {$=$};

	\node[draw=red,fill=white] at (13,0.3) {$S$};
	\node[draw=black,fill=white] at (14,0.3) {$A$};
	
	\draw[red,thick] (13,0) -- (13,-3);
	\draw (14,0) -- (14,-3);
	\draw plot [smooth] coordinates {(11,-1) (11.07,-.7) (11.2,-0.46) (11.5,-0.3) (11.8,-0.46) (11.93,-.7) (12,-1)};
	\draw plot [smooth] coordinates {(9,-2) (9.07,-1.7) (9.2,-1.46) (9.5,-1.3) (9.8,-1.46) (9.93,-1.7) (10,-2)};

	\draw plot [smooth] coordinates {(9,-2) (9.2,-2.3) (10.8,-2.6) (11,-3)};
	\draw plot [smooth] coordinates {(10,-2) (9.9,-2.3) (9.1,-2.7) (9,-3)};
	\draw plot [smooth] coordinates {(11,-2) (10.9,-2.3) (10.1,-2.7) (10,-3)};
	\draw (11,-1) -- (11,-2);
	\draw (12,-1) -- (12,-3);

	\node[draw=black,fill=white] at (11,-1) {$A$};
	\node[draw=black,fill=white,inner sep=2pt] at (12,-1) {$D(A)$};
	\node[draw=black,fill=white] at (9,-1.9) {$S$};
	\node[draw=black,fill=white] at (10,-1.9) {$S$};

	\node[draw=black,fill=white] at (9,-3.3) {$S$};
	\node[draw=black,fill=white] at (10,-3.3) {$A$};
	\node[draw=black,fill=white] at (11,-3.3) {$S$};
	\node[draw=black,fill=white,inner sep=2pt] at (12,-3.3) {$D(A)$};
	\node[draw=red,fill=white] at (13,-3.3) {$S$};
	\node[draw=black,fill=white] at (14,-3.3) {$A$};
\end{tikzpicture}
\caption{\label{fig:muSA} The graphical representation of the identity used for $\mu_{SA} \otimes I_{SA}$.}
\end{center}
\end{figure}
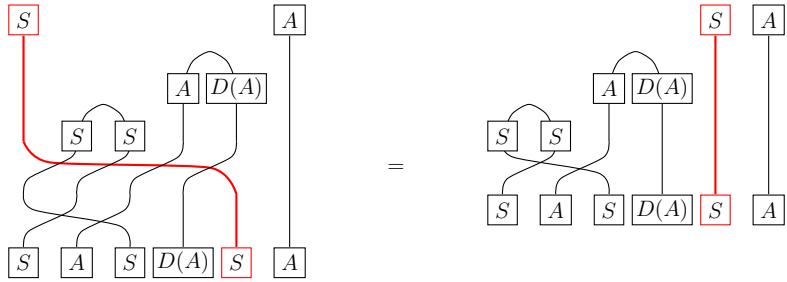
It is easy to see that
\begin{align} \label{eqn:permutation}
&(KK(\sigma_{S, S^2A\otimes D(A)})\otimes I_A)\hat{\otimes}(KK(\sigma_{S, SA})\otimes I_{D(A)\otimes SA})\hat{\otimes}\\
&(I_S\otimes KK(\sigma_{A\otimes SD(A), S})\otimes I_A)\hat{\otimes}(I_{S^2}\otimes KK(\sigma_{A, S})\otimes I_{D(A)\otimes A})\notag\\
=&KK(\sigma_{S^2, S})\otimes I_{A\otimes D(A)\otimes A}\notag\\
=&(-1)^2I_{S^3A\otimes D(A)\otimes A}. \notag
\end{align}
The graphical representation of this identity is shown in Fig.~\ref{fig:permutation}.
\begin{figure}[htp]
\begin{center}
\begin{tikzpicture}[scale=0.7,every node/.style={scale=0.7}]
	\node[draw=red,fill=white] at (0,0.3) {$S$};
	\node[draw=black,fill=white] at (1,0.3) {$S$};
	\node[draw=black,fill=white] at (2,0.3) {$S$};
	\node[draw=blue,fill=white] at (3,0.3) {$A$};
	\node[draw=teal,fill=white,inner sep=2pt] at (4,0.3) {$D(A)$};
	\node[draw=black,fill=white] at (5,0.3) {$A$};
	\draw[red,thick] plot [smooth,tension=.4] coordinates {(0,0) (0.5,-0.4) (3.5,-0.5) (4,-1) (4,-2)};
	\draw plot [smooth] coordinates {(1,0) (0.9,-.3) (0.1,-.7) (0,-1)};
	\draw plot [smooth] coordinates {(2,0) (1.9,-.3) (1.1,-.7) (1,-1)};
	\draw[blue,thick] plot [smooth] coordinates {(3,0) (2.9,-.3) (2.1,-.7) (2,-1)};
	\draw[teal,thick] plot [smooth] coordinates {(4,0) (3.9,-.3) (3.1,-.7) (3,-1)};
	\draw (5,0) -- (5,-4);
	\draw plot [smooth] coordinates {(0,-1) (0.2,-1.3) (1.8,-1.6) (2,-2)};
	\draw plot [smooth] coordinates {(1,-1) (0.9,-1.3) (0.1,-1.7) (0,-2)};
	\draw[blue,thick] plot [smooth] coordinates {(2,-1) (1.9,-1.3) (1.1,-1.7) (1,-2)};
	\draw[teal,thick] (3,-1) -- (3,-2);
	\draw plot [smooth] coordinates {(0,-2) (0,-4)};
	\draw[blue,thick] plot [smooth] coordinates {(1,-2) (1.1, -2.3) (1.89,-2.7) (2,-3)};
	\draw plot [smooth] coordinates {(2,-2) (2.1, -2.3) (2.89,-2.7) (3,-3)};
	\draw[teal,thick] plot [smooth] coordinates {(3,-2) (3.1, -2.3) (3.89,-2.7) (4,-3)};
	\draw[red,thick] plot [smooth] coordinates {(4,-2) (3.6,-2.3) (1.4,-2.6) (1,-3)};
	\draw[red,thick] (1,-3) -- (1,-4);
	\draw[teal,thick] (4,-3) -- (4,-4);
	\draw[blue,thick] plot [smooth] coordinates {(2,-3) (2.1, -3.3) (2.89,-3.7) (3,-4)};
	\draw plot [smooth] coordinates {(3,-3) (2.89,-3.3) (2.1, -3.7) (2,-4)};
	\node[draw=black,fill=white] at (0,-4.3) {$S$};
	\node[draw=red,fill=white] at (1,-4.3) {$S$};
	\node[draw=black,fill=white] at (2,-4.3) {$S$};
	\node[draw=blue,fill=white] at (3,-4.3) {$A$};
	\node[draw=teal,fill=white,inner sep=2pt] at (4,-4.3) {$D(A)$};
	\node[draw=black,fill=white] at (5,-4.3) {$A$};
	
	\node at (7,-0.5) {$=$};
	
	\node[draw=black,fill=white] at (9,0.3) {$S$};
	\node[draw=black,fill=white] at (10,0.3) {$S$};
	\node[draw=black,fill=white] at (11,0.3) {$S$};
	\node[draw=black,fill=white] at (12,0.3) {$A$};
	\node[draw=black,fill=white,inner sep=2pt] at (13,0.3) {$D(A)$};
	\node[draw=black,fill=white] at (14,0.3) {$A$};
	
	\draw plot [smooth] coordinates {(9,0) (9.1,-0.3) (9.87, -0.7) (10,-1)};
	\draw plot [smooth] coordinates {(10,0) (10.1,-0.3) (10.87, -0.7) (11,-1)};
	\draw plot [smooth] coordinates {(11,0) (10.75,-0.3) (9.2,-0.7) (9,-1)};
	\draw (12,0) -- (12,-1);
	\draw (13,0) -- (13,-1);
	\draw (14,0) -- (14,-1);
	
	\node[draw=black,fill=white] at (9,-1.3) {$S$};
	\node[draw=black,fill=white] at (10,-1.3) {$S$};
	\node[draw=black,fill=white] at (11,-1.3) {$S$};
	\node[draw=black,fill=white] at (12,-1.3) {$A$};
	\node[draw=black,fill=white,inner sep=2pt] at (13,-1.3) {$D(A)$};
	\node[draw=black,fill=white] at (14,-1.3) {$A$};
\end{tikzpicture}
\caption{\label{fig:permutation} The graphical representation of equation \eqref{eqn:permutation}.}
\end{center}
\end{figure}
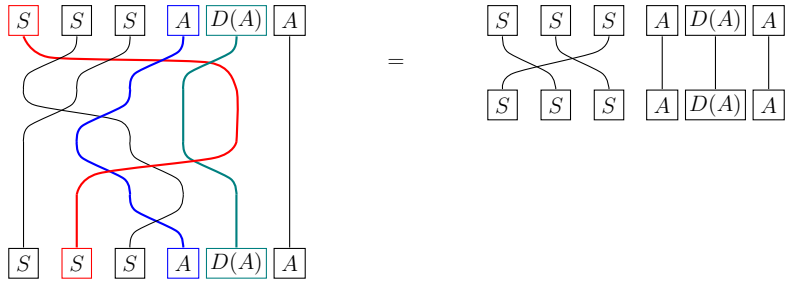

It is straightforward to check that 
\[
	(\mu_{SA}\otimes I_{SA})\hat{\otimes}(I_{SA}\otimes \nu_{SA})=I_{SA}\ .
\]
The other equation in the adjunction formula is verified similarly.
\end{proof}

\begin{lem}\label{cii}
For the $*$-homomorphisms $f : B\to A$ and $D(f) : D(A)\to D(B)$,
the morphism $\mu \in KK(\mathbb{C}, C_f\otimes (SC_{D(f)}))$ satisfies
\[D_{\mu, \nu_B}(KK(e(f)))=d(D(f)):=(\beta\otimes I_{D(B)})\hat{\otimes}(I_S\otimes KK(i({D(f)}))),\]
\[\mu_{SA}\hat{\otimes}(KK(i(f))\otimes I_{SD(A)})=\mu\hat{\otimes}(I_{C_f}\otimes I_S\otimes KK(e({D(f)}))).\]
\end{lem}
\begin{proof}
Recall the following notation from Lem. \ref{K}
\[(F, a, b)\in K_{f, D(f)}\subset C_0((0, 1]^2, A\otimes D(B))\oplus C_0((0, 1], A\otimes D(A))\oplus C_0((0, 1], B\otimes D(B)).\]
It is straightforward to check that with this identification
\begin{align*}
&e(f)\otimes {\rm id}_{C_{D(f)}} : K_{f, D(f)}\ni (F, a, b)\mapsto \sigma_{S, B}\otimes{\rm id}_{D(B)}(b)\in B\otimes SD(B)\subset B\otimes C_{D(f)},\\
&{\rm id}_{C_f}\otimes e({D(f)}) : K_{f, D(f)}\ni (F, a, b)\mapsto a\in SA\otimes D(A)\subset C_f\otimes D(A).
\end{align*}
Now we have the following commutative diagrams
\[
\begin{tikzcd}
B\otimes C_{D(f)} & C_f\otimes C_{D(f)} \ar[l,"e(f)\otimes{\rm id}"] \\
B\otimes SD(B)\ar[u,"{\rm id}\otimes i({D(f)})"] & K_{f, D(f)} \ar[u] \ar[d,"r^*"] \\
SB\otimes D(B)\ar[u,"\sigma_{S, B}\otimes {\rm id}"] & C_\varphi\ar[dl,"e(\varphi)"] \\
(SA\otimes D(A))\oplus(SB\otimes D(B)) , \ar[u,"Pr_2"] &
\end{tikzcd}
\]
\[
\begin{tikzcd}
C_f\otimes D(A) & C_f\otimes C_{D(f)} \ar[l,"{\rm id} \otimes e({D(f)})"] \\
SA\otimes D(A) \ar[u,"i(f)\otimes{\rm id}"] & K_{f, D(f)} \ar[u] \ar[d,"r^*"] \\
(SA\otimes D(A))\oplus (SB\otimes D(B))\ar[u,"Pr_1"] & C_\varphi .\ar[l,"e(\varphi)"]
\end{tikzcd}
\]
By the first diagram and the construction of $\mu$,
one verifies 
\begin{align*}
&\mu\hat{\otimes}(KK(e(f))\otimes I_{SC_{D(f)}})\\
=&\beta_S\hat{\otimes}(I_S\otimes s\mu)\hat{\otimes}(I_S\otimes KK(Pr_2\circ e(\varphi)))\hat{\otimes}(I_S\otimes KK(({\rm id}_B\otimes i({D(f)}))\circ (\sigma_{S, B}\otimes {\rm id}_{D(B)})))\\
&\hat{\otimes}(KK(\sigma_{S, B})\otimes I_{C_{D(f)}})\\
=&\beta_S\hat{\otimes}(I_{S^2}\otimes\mu_B)\hat{\otimes}(KK(\sigma_{S^2, B})\otimes I_{D(B)})\hat{\otimes}(I_{B\otimes S}\otimes KK(i({D(f)}))),
\end{align*}
and a direct computation yields
\begin{align*}
&\beta_S\hat{\otimes}(I_{S^2}\otimes\mu_B)\hat{\otimes}(KK(\sigma_{S^2, B})\otimes I_{D(B)})\hat{\otimes}(I_{B\otimes S}\otimes KK(i({D(f)})))\\
=&\mu_B\hat{\otimes}(\beta_S\otimes I_{B\otimes D(B)})\hat{\otimes}(KK(\sigma_{S^2, B})\otimes I_{D(B)})\hat{\otimes}(I_{B\otimes S}\otimes KK(i({D(f)})))\\
=&\mu_B\hat{\otimes}(I_{B\otimes D(B)}\otimes \beta_S)\hat{\otimes}(I_B\otimes KK(\sigma_{D(B), S^2}))\hat{\otimes}(I_{B\otimes S}\otimes KK(i({D(f)}))).
\end{align*}
Thus we have
\begin{align*}
D_{\mu, \nu_B}(KK(e(f)))=&(I_{D(B)}\otimes (\mu\hat{\otimes}(KK(e(f))\otimes I_{SC_{D(f)}})))\hat{\otimes}(\nu_B\otimes I_{SC_{D(f)}})\\
=&(I_{D(B)}\otimes\mu_B)\hat{\otimes}(I_{D(B)\otimes B}\otimes ((I_{D(B)}\otimes \beta_S)\hat{\otimes}KK(\sigma_{D(B), S^2})\hat{\otimes}(I_S\otimes KK(i({D(f)})))))\\
&\hat{\otimes}(\nu_B\otimes I_{SC_{D(f)}})\\
=&(I_{D(B)}\otimes \beta_S)\hat{\otimes}KK(\sigma_{D(B), S^2})\hat{\otimes}(I_S\otimes KK(i({D(f)})))\\
=&(\beta_S\otimes I_{D(B)})\hat{\otimes}(I_S\otimes KK(i({D(f)})))\\
=&d({D(f)})
\end{align*}
Similarly, the second diagram shows
\begin{align*}
&\mu\hat{\otimes}(I_{C_f\otimes S}\otimes KK(e(D(f)))\\
=&\beta_S\hat{\otimes}(I_{S^2}\otimes\mu_A)\hat{\otimes}(I_S\otimes KK(i(f))\otimes I_{D(A)})\hat{\otimes}(KK(\sigma_{S, C_f})\otimes I_{D(A)})\\
=&\mu_{SA}\hat{\otimes}(KK(i(f))\otimes I_{SD(A)}). \qedhere
\end{align*}
\end{proof}
\begin{proof}[{Proof of Thm. \ref{Mt1}}]
By Lem. \ref{ci} and Lem. \ref{cii},
it is enough to show the bijectivity of the natural transformation
\[\mu\hat{\otimes} : KK(P\otimes C_f, Q)\to KK(P, Q\otimes (SC_{D(f)})).\]
Since $P\otimes C_f$ (resp. $Q\otimes SC_{D(f)}$) is identified with $C_{{\rm id}_P\otimes f}$ (resp. $SC_{D(f)\otimes{\rm id}_Q}$),
the Puppe sequence (see \cite[{Thm. 19.4.3.}]{B}) gives the following horizontal exact sequences :
\[
\begin{tikzcd}[column sep=1.6cm]
KK_X(P\otimes SB, Q)\ar[d,"\mu_{SB}\hat{\otimes}"] & KK_X(P\otimes SA, Q) \ar[d,"\mu_{SA}\hat{\otimes}"] \ar[l,"Sf\hat{\otimes}"] & KK_X(P\otimes C_f, Q) \ar[l,"i(f)\hat{\otimes}"] \ar[d,"\mu\hat{\otimes}"] & \ar[l] \\
KK_X(P, Q\otimes SD(B)) & KK_X(P, Q\otimes SD(A)) \ar[l,"\hat{\otimes}SD(f)"] & KK_X(P, Q\otimes SC_{D(f)}) \ar[l,"\hat{\otimes}Se({D(f)})"] & \ar[l]
\end{tikzcd}
\]
\[
\begin{tikzcd}[column sep=1.6cm]
\mbox{} & KK_X(P\otimes B, Q) \ar[l,dash,"e(f)\hat{\otimes}"] \ar[d,"\mu_B\hat{\otimes}"] & KK_X(P\otimes A, Q) \ar[l,"f\hat{\otimes}"] \ar[d,"\mu_A\hat{\otimes}"] \\
\mbox{} & KK_X(P, Q\otimes D(B)) \ar[l,dash,"\hat{\otimes}d({D(f)})"] & KK_X(P, Q\otimes D(A)).\ar[l,"\hat{\otimes}D(f)"]
\end{tikzcd}
\]
Note that $$KK(SD(f))=(-1)^2I_S\otimes D_{\mu_B, \nu_A}(KK(f))=D_{\mu_{SB}, \nu_{SA}}(KK(Sf))$$ (c.f. Proof of Lem. \ref{sl}).
In the above diagram,
the two squares in the middle commute by Lem. \ref{cii},
and the left and right square commute by the definition of the dual morphisms.
The Five-lemma shows the bijectivity of $\mu\hat{\otimes}$.
\end{proof}
\begin{rem}\label{nu}
Applying the exact sequence in the above proof for $P=SC_{D(f)}, Q=\mathbb{C}$,
Lem.~\ref{ci} implies that $\mu$ and $\nu:=(\mu\hat{\otimes})^{-1}(I_{SC_{D(f)}})$ are duality classes
and one has $D_{\mu_{SA}, \nu}(KK(i(f)))=I_S\otimes KK(e({D(f)}))$.
\end{rem}
\begin{cor}\label{Mt1c}
Let $SA\xrightarrow{i} C\xrightarrow{e} B\xrightarrow{f} A$ be an exact triangle of separable nuclear UCT C*-algebras with finitely generated K-groups.
Then there exist separable nuclear UCT C*-algebras $D(SA), D(C), D(B), D(A)$ with finitely generated K-groups
and duality classes

\[\mu_{sa}\in KK(\mathbb{C}, SA\otimes D(SA)),\quad \nu_{sa}\in KK(D(SA)\otimes SA, \mathbb{C}),\]
\[\mu_c\in KK(\mathbb{C}, C\otimes D(C)),\quad \nu_c\in KK(D(C)\otimes C, \mathbb{C}),\]
\[\mu_{b}\in KK(\mathbb{C}, B\otimes D(B)),\quad \nu_b\in KK(D(B)\otimes B, \mathbb{C}),\]
\[\mu_a\in KK(\mathbb{C}, A\otimes D(A)),\quad \nu_a\in KK(D(A)\otimes A, \mathbb{C})\]
such that the dual sequence
\[D(SA)\xleftarrow{D_{\mu_{sa}, \nu_c}(i)}D(C)\xleftarrow{D_{\mu_c, \nu_b}(e)}D(B)\xleftarrow{D_{\mu_b, \nu_a}(f)}D(A)\]
is an exact triangle.
\end{cor}
\begin{proof}
It is enough to prove the statement for the mapping cone sequence
\[SA\xrightarrow{i(f)}C_f\xrightarrow{e(f)}B\xrightarrow{f}A.\]
By Thm. \ref{ks} and \cite[{Thm. 2.5}]{Dk},
one has a dual Kirchberg algebra $D_A$ of $A$ (resp. $D_B$ of $B$) with duality classes $(\mu_A, \nu_A)$ (resp. $(\mu_B, \nu_B)$).
By \cite[{Thm. E}]{G},
the dual morphism $D_{\mu_B, \nu_A}(f)$ is represented by a $*$-homomorphism $D_f : D_A\to D_B$,
and Thm. \ref{Mt1} and Rem. \ref{nu} give duality classes $(\mu, \nu)$ for $C_f$ and $D(C_f):=SC_{D_f}$ satisfying
\begin{enumerate}[(i)]
\item $D_{\mu, \nu_B}(KK(e(f)))=(\beta\otimes I_{D_B})\hat{\otimes}(I_S\otimes KK(i(D_f))),$
\item $D_{\mu_{SA}, \nu}(KK(i(f)))=I_S\otimes KK(e(D_f)),$
\end{enumerate}
where we write $\mu_{SA}:=\mu_A\hat{\otimes}(\beta\otimes I_{A\otimes D_A})\hat{\otimes}(KK(\sigma_{S, SA})\otimes I_{D_A})$.
By $S((C_0(0, 1]\otimes D_B)\oplus D_A)\cong (C_0(0, 1]\otimes SD_B)\oplus SD_A$,
we have an isomorphism $\gamma : SC_{D_f}\to C_{SD_f}$.
We write $D(C_f):=C_{SD_f}$ and define $(\mu_c, \nu_c)$ by
\[\mu_c:=\mu\hat{\otimes}(I_{C_f}\otimes KK(\gamma)),\quad \nu_c:=(KK(\gamma^{-1})\otimes I_{C_f})\hat{\otimes}\nu.\]
We write $D(SA):=SD_A, D(B):=S(SD_B), D(A):=S(SD_A)$ and define $(\mu_b, \nu_b), (\mu_{sa}, \nu_{sa})$ as follows:
\[\mu_b:=\mu_B\hat{\otimes}(I_B\otimes ((\beta\hat{\otimes}KK(\sigma_{S, S}))\otimes I_{D_B})),\quad \nu_b:=(((KK(\sigma_{S, S})\otimes\beta^{-1})\otimes I_{D_B})\otimes I_B)\hat{\otimes}\nu_B,\]
\[\mu_{sa}:=\mu_{SA}, \quad \nu_{sa}:=(KK(\sigma_{SD_A, S})\otimes I_A)\hat{\otimes}(\beta^{-1}\otimes I_{(D_A\otimes A)})\hat{\otimes}\nu_A.\]
The equations $\gamma\circ S(i(D_f))\circ\sigma_{S, S}=i(SD_f)$ and $S(e(D_f))\circ\gamma^{-1}=e(SD_f)$ and (i) and (ii) imply
\[D_{\mu_c, \nu_b}(KK(e(f)))=KK(i(SD_f)),\quad D_{\mu_{sa}, \nu_c}(KK(i(f)))=KK(e(SD_f)).\]
Now we define $(\mu_a, \nu_a)$ by
\[\mu_a:=\mu_A\hat{\otimes}(I_A\otimes (\beta\otimes I_{D_A})),\quad \nu_a:=((\beta^{-1}\otimes I_{D_A})\otimes I_A)\hat{\otimes}\nu_A\]
so that the equation $D_{\mu_b, \nu_a}(KK(f))=-KK(S(SD_f))$ holds.
Now we obtain
\[
\begin{tikzcd}[column sep=2.1cm]
D(SA) \ar[d,equal] & D(C_f) \ar[d,equal] \ar[l,"D_{\mu_{sa}, \nu_c}(i(f))"] & D(B)\ar[d,equal]\ar[l,"D_{\mu_c, \nu_b}(e(f))"] & D(A)\ar[l,"D_{\mu_b, \nu_a}(f)"] \ar[d,equal] \\
SD_A & C_{SD_f} \ar[l,"e(SD_f)"] & S(SD_B)\ar[l,"i(SD_f)"] & S(SD_A),\ar[l,"-S(SD_f)"]
\end{tikzcd}
\]
where the bottom sequence is an exact triangle.
\end{proof}

We will use the following corollary in Section \ref{skD}.
\begin{cor}\label{Sed}
Let $f : B\to A$ be a $*$-homomorphism between dualizable algebras, and let $D, D(B)$ be dual algebras of $C_f$ and $B$, respectively, with duality classes
\[\mu_C\in KK(\mathbb{C}, C_f\otimes D),\quad \nu_C\in KK(D\otimes C_f, \mathbb{C}),\]
\[\mu_B\in KK(\mathbb{C}, B\otimes D(B)),\quad \nu_B\in KK(D(B)\otimes B, \mathbb{C}).\]
Assume that there is a $*$-homomorphism $g : D(B)\to D$ satisfying $$D_{\mu_C, \nu_B}(KK(e(f)))=KK(g).$$
Then, there exists a duality class $\mu\in KK(\mathbb{C}, C_g\otimes A)$ satisfying
\[D_{\mu, \sigma_{B, D(B)}\hat{\otimes}\nu_B}(KK(e(g)))=KK(f).\]
\end{cor}
\begin{proof}
The assumption implies
\[D_{\mu_B\hat{\otimes}\sigma_{B, D(B)}, \sigma_{C_f, D}\hat{\otimes}\nu_C}(KK(g))=KK(e(f)),\]
and Thm. \ref{Mt1} gives a duality class $\bar{\mu}\in KK(\mathbb{C}, C_g\otimes SC_{e(f)})$ satisfying
\[D_{\bar{\mu}, \sigma_{B, D(B)}\hat{\otimes}\nu_B}(KK(e(g)))=d(e(f)):=(\beta_S\otimes I_B)\hat{\otimes}(I_S\otimes KK(i(e(f)))).\]
One may identify $C_{e(f)}$ with the algebra
\[\{(b(t), a(s))\in (C_0(0, 1]\otimes B)\oplus (C_0(0, 1]\otimes A)\;|\; f(b(1))=a(1)\}.\]
The exact sequence
\[0\to SA\to C_{e(f)}\to C_0(0, 1]\otimes B\to 0\]
shows that the inclusion $SA\hookrightarrow C_{e(f)}$ is a KK-equivalence making the following diagram commute
\[
\begin{tikzcd}[column sep=2cm]
SB \ar[r,"KK(i(e(f)))"] \ar[d, equal] & C_{e(f)} \\
SB \ar[r,"-I_S\otimes KK(f)"] & SA, \ar[u]
\end{tikzcd}
\]
and a KK-equivalence $\gamma \in KK(A, SC_{e(f)})^{-1}$ defined by
\[\gamma : A\xrightarrow{-\beta_S\otimes I_A}S^2A\hookrightarrow SC_{e(f)}\]
satisfies
\[d(e(f))\hat{\otimes}\gamma^{-1}=KK(f).\]
Thus, the duality class
\[\mu:=\bar{\mu}\hat{\otimes}(I_{C_g}\otimes \gamma^{-1})\in KK(\mathbb{C}, C_g\otimes A)\]
fulfills the required condition.
\end{proof}


\section{Extensions of C*-algebras and KK-theory}

\subsection{Extension groups}
We first recall some basic facts about extension groups and refer to \cite{B} for reference.
Let $A, B$ be separable, nuclear C*-algebras,
and let $\tau_1, \tau_2 : A\to \mathcal{Q}(B\otimes\mathbb{K})$ be $*$-homomorphisms called Busby invariants.
Two homomorphisms are strongly equivalent (resp. weakly equivalent) if there is a unitary $U\in\mathcal{M}(B\otimes \mathbb{K})$ (resp. $u\in \mathcal{Q}(B\otimes\mathbb{K})$) satisfying $\tau_1={\rm Ad} \pi(U)\circ \tau_2$ (resp. $\tau_1={\rm Ad} u\circ\tau_2$).
The Busby invariants $\tau_1$ and $\tau_2$ are called stably equivalent if there are $*$-homomorphisms $\rho_1, \rho_2 : A\to\mathcal{M}(B\otimes\mathbb{K})$ such that $\tau_1\oplus \pi\circ\rho_1$ and $\tau_2\oplus\pi\circ\rho_2$ are strongly equivalent.
If $\tau : A\to \mathcal{Q}(B\otimes\mathbb{K})$ is injective,
the Busby invariant is called essential,
and it gives an essential extension
\[
\begin{tikzcd}
B\otimes\mathbb{K} \ar[d,equal] \ar[r] & \pi^{-1}(\tau (A)) \ar[d] \ar[r] & A \ar[d,"\tau"]\\
B\otimes\mathbb{K} \ar[r] & \mathcal{M}(B\otimes\mathbb{K})\ar[r,"\pi"] & \mathcal{Q}(B\otimes\mathbb{K}).
\end{tikzcd}
\]
For a unital C*-algebra $A$,
the Busby invariant is called unital if $\tau : A\to \mathcal{Q}(B\otimes\mathbb{K})$ is unital,
and the corresponding extension is called unital extension.
We denote by $Ext(A, B\otimes\mathbb{K})$ the group of stable equivalence classes of the Busby invariants.
It is well-known that the group $Ext(A, B\otimes\mathbb{K})$ can be naturally identified with $KK^1(A, B)$.

Recall that every element of $KK(A, B)$ is represented by a Cuntz pair $[\phi_0, \phi_1]$,
where $\phi_0, \phi_1 : A\to \mathcal{M}(B\otimes\mathbb{K})$ are $*$-homomorphisms satisfying $\phi_0(a)-\phi_1(a)\in B\otimes\mathbb{K}$,
and one has $KK(f)=[f\otimes e, 0]$ for a $*$-homomorphism $f : A\to B$.
The pull-back of the extension $$SB\otimes\mathbb{K}\to C_0(0, 1]\otimes B\otimes\mathbb{K}\xrightarrow{ev_1} B\otimes\mathbb{K}$$ by $f\otimes e : A\to B\otimes\mathbb{K}$ gives an element of $Ext(A, SB\otimes\mathbb{K})$,
and this extends to the following natural isomorphism
\begin{equation} \label{eqn:Blackadar-19.2.6}
	\eta_{A, B} : KK(A, B)\ni [\phi_0, \phi_1]\mapsto [\tau_{[\phi_0, \phi_1]}]\in Ext(A, SB\otimes\mathbb{K}),
\end{equation}
where the Busby invariant is defined by (see \cite[{19.2.6.}]{B} and the proof in the appendix)
\begin{equation} \label{eqn:Blackadar-def}
	\tau_{[\phi_0, \phi_1]}(a)=\pi (t\phi_0(a)+(1-t)\phi_1(a))\in \mathcal{Q}(SB\otimes\mathbb{K}).
\end{equation}
Here, we identify the function $[0, 1]\ni t\mapsto t\phi_0(a)+(1-t)\phi_1(a) \in\mathcal{M}(B\otimes\mathbb{K})$ with an element of $C[0, 1]\otimes\mathcal{M}(B\otimes\mathbb{K})\subset \mathcal{M}(SB\otimes\mathbb{K})$.
\begin{rem}
Note that the definition of $\tau_{[\phi_0, \phi_1]}$ is slightly different from the one in \cite[{19.2.6.}]{B},
and the difference comes from the definition of  mapping cone algebra.
To construct the mapping cone, we use $C_0(0, 1]$ and the reference uses $C_0[0, 1)$.
\end{rem}
\begin{rem}\label{es}
The extension $\tau_{[S\phi_0, S\phi_1]}$ is identified with 
\[SA\xrightarrow{{\rm id}_S\otimes \tau_{[\phi_0, \phi_1]}}S\mathcal{Q}(SB\otimes\mathbb{K})\subset\mathcal{Q}(S^{2}B\otimes \mathbb{K})\xrightarrow{\sigma_{S, S}\otimes {\rm id}_{B\otimes\mathbb{K}}}\mathcal{Q}(S^{2}B\otimes\mathbb{K}).\]
Thus,
one has $\eta_{SA, SB}(I_S\otimes x)=-I_S\otimes \eta_{A, B}(x)$ for $x\in KK(A, B)$.
\end{rem}

Since $A$ is nuclear,
every essential extension 
\[SB\otimes\mathbb{K}\to E\xrightarrow{\pi_E} A\]
is semi-split and the morphism $KK(j(E))\in KK(SB\otimes\mathbb{K}, C_{\pi_E})$ is known to be a KK-equivalence (see \cite[{Thm. 19.5.5.}]{B}).
\begin{lem}[{c.f. \cite[{Lem. 19.5.6, Thm. 19.5.7.}]{B}}]\label{li}
Let $E:=\pi^{-1}(\tau(A))\subset \mathcal{M}(SB\otimes\mathbb{K})$ be an essential  extension of $A$ with the quotient map $\pi_E : E\ni x\mapsto \tau^{-1}(\pi(x)) \in A$.
Then, we have \[-I_S\otimes (\eta_{A, B}^{-1}([\tau])\hat{\otimes}KK(e))=KK(i({\pi_E}))\hat{\otimes}KK(j(E))^{-1}.\]
In particular,
the following sequence is an exact triangle:
\[SA\xrightarrow{-I_S\otimes (\eta_{A, B}^{-1}([\tau])\hat{\otimes}KK(e))}SB\otimes\mathbb{K}\to E\xrightarrow{\pi_E} A.\]
\end{lem}
\begin{proof}
For an element  $x\in KK(-, S^iB)$,
we use the following short-hand notation $$x\hat{\otimes}KK(e):= x\hat{\otimes}(I_{S^iB}\otimes KK(e))\in KK(-, S^iB\otimes \mathbb{K}).$$
One has the following diagram
\[
\begin{tikzcd}[column sep=3.2cm]
SA \ar[d,equal] \ar[r,"-I_S \otimes (\eta_{A, B}^{-1}({[\tau]})\hat{\otimes}KK(e))"] & SB \otimes\mathbb{K} \ar[d,"j(E)"] \ar[r] & E\ar[d,equal]\ar[r,"\pi_E"] & A \ar[d,equal]\\
SA \ar[r,"i({\pi_E})"] & C_{\pi_E} \ar[r,"e({\pi_E})"] & E\ar[r,"\pi_E"] & A,
\end{tikzcd}
\]
where the middle and right squares commute.
The surjective map 
$$C_0(0, 1]\otimes E\ni x(t)\mapsto (\pi_E(x(t)), x(1))\in C_{\pi_E}$$
gives an essential extension
$$S(SB\otimes\mathbb{K})\to C_0(0, 1]\otimes E\to C_{\pi_E}.$$
For the pull-back $j(E)^* : Ext(C_{\pi_E}, S(SB\otimes\mathbb{K}))\to Ext(SB\otimes\mathbb{K}, S(SB\otimes\mathbb{K}))$,
one has 
\begin{align*}
&j(E)^*([S(SB\otimes\mathbb{K})\to C_0(0, 1]\otimes E\to C_{\pi_E}])\\
=&[S(SB\otimes\mathbb{K})\to C_0(0, 1]\otimes (SB\otimes\mathbb{K})\to SB\otimes\mathbb{K}]\\
=&[\tau_{[{\rm id}_{SB\otimes \mathbb{K}}, 0]}]\\
=&\eta_{SB\otimes\mathbb{K}, SB}([{\rm id}_{SB\otimes\mathbb{K}}, 0]) \in Ext(SB\otimes\mathbb{K}, S(SB\otimes\mathbb{K})).
\end{align*}
Since $[{\rm id}_{SB\otimes\mathbb{K}}, 0]\hat{\otimes} KK(e)$ is represented by the Cuntz pair $[\phi, 0]$ with
\[\phi :  SB\otimes \mathbb{K}= SB\otimes e\otimes \mathbb{K}\subset SB\otimes \mathbb{K}\otimes \mathbb{K},\]
one has 
\begin{align*}
[{\rm id}_{SB\otimes\mathbb{K}}, 0]\hat{\otimes} KK(e)=&[\phi, 0]\\
=&[{\rm id}_{SB\otimes\mathbb{K}}\otimes e, 0]\\
=& I_{SB\otimes\mathbb{K}}+({\rm degenerate \;module})\\
=&I_{SB\otimes \mathbb{K}},
\end{align*}
and the naturality of $j(E)\hat{\otimes}-$ and $\eta_{-, -}$ implies
\begin{align*}
&\eta_{C_{\pi_E}, SB}(KK(j(E))^{-1}\hat{\otimes}KK(e)^{-1})\\
=&[S(SB\otimes\mathbb{K})\to C_0(0, 1]\otimes E\to C_{\pi_E}]\in Ext(C_{\pi_E}, S(SB\otimes\mathbb{K})).
\end{align*}
One has
\begin{align*}
&i({\pi_E})^*([S(SB\otimes\mathbb{K})\to C_0(0, 1]\otimes E\to C_{\pi_E}])\\
=&[S(SB\otimes\mathbb{K})\to SE\xrightarrow{S\pi_E} SA]\\
=&[{\rm id}_S\otimes \tau]\in Ext(SA, S(SB\otimes\mathbb{K})).
\end{align*}
Rem. \ref{es} and the naturality of $\eta_{-,-}$ show
\[
KK(i({\pi_E}))\hat{\otimes}KK(j(E))^{-1}\hat{\otimes}KK(e)^{-1}=\eta_{SA, SB}^{-1}(I_S\otimes [\tau])=-I_S\otimes \eta_{A, B}^{-1}([\tau]). \qedhere
\]
\end{proof}
\subsection{Strong extension groups}
For a unital C*-algebra $A$,
we identify the mapping cone $C_{u_A}$ with the algebra 
\[\{a(t)\in C_0(0, 1]\otimes A \; |\; a(1)\in\mathbb{C}1_A\}.\]
For a separable C*-algebra $B$ and a unital, separable,  nuclear C*-algebra $A$,
G. Skandalis \cite{SK} introduces the strong extension group
\[Ext_s(A, B\otimes\mathbb{K}):=\{\tau : A\to \mathcal{Q}(B\otimes\mathbb{K})\;|\; \tau : {\rm unital\; extension}\}/_\sim\]
where $\tau_1\sim\tau_2$ means that there exist unital $*$-homomorphisms $\rho_1, \rho_2 : A\to \mathcal{M}(B\otimes\mathbb{K})$ and a unitary $U\in \mathbb{M}_2(\mathcal{M}(B\otimes\mathbb{K}))$ satisfying
\[\tau_1\oplus \pi\circ\rho_1={\rm Ad}U\circ(\tau_2\oplus\pi\circ\rho_2).\]
He shows the following theorem.
\begin{thm}[{\cite[Thm. 2.3., Cor. 2.4.]{SK}}]\label{SKase}
\begin{enumerate}
\item{}There exists a 6-term exact sequence
\[
\begin{tikzcd}
K_0(B) \ar[r] & Ext_s(A, B\otimes\mathbb{K}) \ar[r] & Ext(A, B\otimes \mathbb{K}) \ar[d] \\
Ext(A, SB\otimes\mathbb{K}) \ar[u] & Ext_s(A, SB\otimes\mathbb{K}) \ar[l] & K_0(SB).\ar[l]
\end{tikzcd}
\]
\item{} We have a natural isomorphism
\[Ext_s(A, B\otimes\mathbb{K})\to Ext (C_{u_A}, SB\otimes\mathbb{K})\]
sending $[B\otimes\mathbb{K}\to E\to A]$ to $[SB\otimes\mathbb{K}\to C_{u_E}\to C_{u_A}]$.
\end{enumerate}
\end{thm}
Following \cite{M0},
we denote by $Ext_s(A)$ the group of strong equivalence classes of unital essential extensions $\tau : A\to \mathcal{Q}(\mathbb{K})$.
Note that Voiculescu's theorem shows $Ext_s(A)=Ext_s(A, \mathbb{K})$,
and the second statement of the above theorem implies $Ext_s(A)\cong KK(C_{u_A}, \mathbb{C})$.

In the rest of this subsection,
we take a closer look at the above theorem in the case of $B=\mathbb{C}$ to fix notation.
For the unital extension $\tau$ with  $E:=\pi^{-1}(\tau(A))\subset \mathcal{M}(\mathbb{K})$,
we write $[\tau]_s=[E]_s\in Ext_s(A)$.
The nuclearlity of $A$ provides a unital completely positive lifting $L_\tau : A\to \mathcal{M}(\mathbb{K})$ of $\tau$,
and the composition of the following maps
\[C_{u_A}\ni a(t) \mapsto L_\tau (a(t))\in C_0(0, 1]\otimes\mathcal{M}(\mathbb{K}),\]
\[C_0(0, 1]\otimes\mathcal{M}(\mathbb{K})\subset \mathcal{M}(S\mathbb{K})\xrightarrow{\pi} \mathcal{Q}(S\mathbb{K})\]
define an essential extension $c(\tau) : C_{u_A}\to\mathcal{Q}(S\mathbb{K})$ with  $\pi^{-1}(c(\tau)(C_{u_A}))=C_{u_E}\subset C_0(0, 1]\otimes\mathcal{M}(\mathbb{K})$.
It is easy to check that $c(\tau)$ is independent of the choice of the unital completely positive lift $L_\tau$.
This construction induces a group homomorphism
\[m_A : Ext_s (A)\ni [\tau]_s=[E]_s\mapsto [c(\tau)]=[S\mathbb{K}\to C_{u_E}\to C_{u_A}]\in Ext(C_{u_A}, S\mathbb{K}).\]
For the above essential extensions $E, C_{u_E}$,
we write $\pi_E :=\tau^{-1}\circ\pi, \; \pi_{C_{u_E}}:=c(\tau)^{-1}\circ \pi$,
and $\pi_{C_{u_E}}={\rm id}_{C_0(0, 1]}\otimes \pi_E$ holds by definition.

We denote by $Ext_w(A)$ the group of weak equivalence classes of unital essential extensions of $A$ by $\mathbb{K}$.
For a general $*$-homomorphism $\sigma : A\to \mathcal{Q}(\mathbb{K})$,
we may assume that the projection $\tau (1_A) \in \mathcal{Q}(\mathbb{K})$ is properly infinite and is Murray--von Neumann equivalent to $1_{\mathcal{Q}(\mathbb{K})}$.
Combining the above and  Voiculescu's absorption theorem (see \cite[{Chap. 15.12.}]{B}),
it is well-known that the natural map
\[Ext_w(A)\ni[\tau]_w\mapsto [\tau]\in Ext(A, \mathbb{K})\]
is an isomorphism.

We denote by $\tau_A:=\pi\circ\rho$ an essential unital trivial extension with an injective unital $*$-homomorphism $\rho : A\to \mathcal{M}(\mathbb{K})$,
and the kernel of the natural surjection
\[q_A : Ext_s(A)\ni [\tau]_s\mapsto [\tau]_w\in  Ext_w(A)\]
is the image of the following group homomorphism
\[\iota_A : K_1(\mathcal{Q}(\mathbb{K}))\ni [u]_1\mapsto [{\rm Ad} u\circ \tau_A]_s\in Ext_s(A),\]
where we identify $K_1(\mathcal{Q}(\mathbb{K}))$ with $U(\mathcal{Q}(\mathbb{K}))/\sim_{h}$ (see \cite[{Lem. 1.2.}]{M0}).
Let $V$ be an isometry such that $e:=1-VV^*\in \mathbb{K}$ is a rank 1 projection,
and we identify $[\pi(V)]_1\in K_1(\mathcal{Q}(\mathbb{K}))$ with $-1={\rm Ind}(V)=-[e]_0\in\mathbb{Z}=K_0(\mathbb{K})$ as in \cite{M}.
For the map $e : \mathbb{C}\to \mathbb{C}e\subset\mathbb{K}$ and a Cuntz pair $[e, 0]=I_\mathbb{C}\in KK(\mathbb{C}, \mathbb{C})$,
we identify $Ext(\mathbb{C}, S\mathbb{K})$ with $K_1(\mathcal{Q}(\mathbb{K}))$ by 
\[-I : K_1(\mathcal{Q}(\mathbb{K}))\ni [\pi (V)]_1\mapsto [\tau_{[e, 0]}] \in Ext(\mathbb{C}, S\mathbb{K}).\]

The groups $Ext_s(A)$ can be computed from KK-groups as shown in the following theorem.
\begin{thm}[{c.f. \cite[Cor. 2.4.]{SK}}]\label{mcs}
For a unital C*-algebra $A$, the map $m_A$ is an isomorphism making the following diagram commute
\[
\begin{tikzcd}
K_1(\mathcal{Q}(\mathbb{K})) \ar[d,"-I"] \ar[r,"\iota_A"] & Ext_s(A) \ar[d,"m_A"] \\
Ext(\mathbb{C}, S\mathbb{K}) \ar[r,"e({u_A})^*"] & Ext(C_{u_A}, S\mathbb{K}).
\end{tikzcd}
\]
\end{thm}
\begin{lem}[{c.f. \cite[{Sec. 2.2.}]{M}}]\label{mc1}
The following is an exact sequence:
\[K_1(\tau_A(A)'\cap\mathcal{Q}(\mathbb{K}))\to K_1(\mathcal{Q}(\mathbb{K}))\xrightarrow{\iota_A}Ext_s(A)\xrightarrow{q_A}Ext_w(A)\to 0,\]
where the map $K_1(\tau_A(A)'\cap\mathcal{Q}(\mathbb{K}))\to K_1(\mathcal{Q}(\mathbb{K})$ is induced by the inclusion $\tau_A(A)'\cap\mathcal{Q}(\mathbb{K})\subset\mathcal{Q}(\mathbb{K})$.
\end{lem}
\begin{proof}
It is enough to show ${\rm Ker}\; \iota_A \subset {\rm Im} (K_1(\tau_A(A)'\cap\mathcal{Q}(\mathbb{K}))\to K_1(\mathcal{Q}(\mathbb{K})))$.
For a unitary $w\in \mathcal{Q}(\mathbb{K})$ with $[{\rm Ad}w\circ\tau_A]_s=0$,
there is a unitary $U\in\mathcal{M}(\mathbb{K})$ satisfying ${\rm Ad}\pi(U)w\circ \tau_A=\tau_A$,
and $[w]_1=[\pi(U)w]_1\in {\rm Ker}\; \iota_A$ lies in the image of $K_1(\tau_A(A)'\cap\mathcal{Q}(\mathbb{K}))\to K_1(\mathcal{Q}(\mathbb{K}))$.
\end{proof}
\begin{lem}\label{mc2}
There is an isomorphism $p_A : K_1(\tau_A(A)'\cap\mathcal{Q}(\mathbb{K}))\to Ext(SA, \mathbb{K})$ making the following square commute
\[
\begin{tikzcd}
K_1(\tau_A(A)'\cap\mathcal{Q}(\mathbb{K})) \ar[r] \ar[d,"p_A"] & K_1(\mathcal{Q}(\mathbb{K}))\ar[d,"-I"] \\
Ext(A, S\mathbb{K})\ar[r,"u_A^*"] & Ext(\mathbb{C}, S\mathbb{K})
\end{tikzcd}
\]
\end{lem}
\begin{proof}
By Paschke duality,
one has an isomorphism
\[P_{\tau_A} : K_1(\tau_A(A)'\cap\mathcal{Q}(\mathbb{K}))\to Ext(SA, \mathbb{K})\]
 which sends $w\in U(\tau_A(A)'\cap\mathcal{Q}(\mathbb{K}))$ to an extension defined by
\[SA\ni f\otimes a\mapsto f(w)\tau_A(a)\in\mathcal{Q}(\mathbb{K}). \]
Thus, one has the commutative square 
\[
\begin{tikzcd}
K_1(\tau_A(A)'\cap\mathcal{Q}(\mathbb{K})) \ar[d,"P_{\tau_A}"] \ar[r] & K_1(\mathcal{Q}(\mathbb{K}))\ar[d,"P_{(\tau_A\circ u_A)}"] \\
Ext (SA, \mathbb{K})\ar[r,"(Su_A)^*"] & Ext (S, \mathbb{K}).
\end{tikzcd}
\]
The generator $-1\in \mathbb{Z}=K_1(\mathcal{Q}(\mathbb{K}))$ is given by $[\pi(V)]_1$,
where $V\in\mathcal{M}(\mathbb{K})$ is an isometry with a rank 1 projection $e:=1-VV^*$.
Multiplying by the appropriate sign $\pm 1$,
one can find a natural isomorphism $\theta_{A, B} : Ext (SA, B\otimes\mathbb{K})\to Ext (A, SB\otimes\mathbb{K})$ such that
\[\theta_{\mathbb{C}, \mathbb{C}}(P_{(\tau_A\circ u_A)}([\pi(V)]))=[\tau_{[e, 0]}]\]
(i.e., $\theta_{\mathbb{C},\mathbb{C}}\circ P_{(\tau_A\circ u_A)}=-I$).
For the isomorphism $p_A:=\theta_{A, \mathbb{C}}\circ P_{\tau_A}$,
the naturality of $\theta_{-, -}$ implies $u_A^*\circ p_A=\theta_{\mathbb{C}, \mathbb{C}}\circ(Su_A)^*\circ P_{\tau_A}$,
and this proves the statement.
\end{proof}
\begin{proof}[{Proof of Thm. \ref{mcs}}]
Let $f_A : Ext_w(A)\to Ext(SA, S\mathbb{K})$ be the composition of the isomorphism $Ext_w(A)\to Ext(A, \mathbb{K})$ and the suspension isomorphism $I_S\otimes -: Ext(A, \mathbb{K})\to Ext(SA, S\mathbb{K})$.
Since $c(\tau)\circ i({u_A})=S\tau : SA\to S\mathcal{Q}(\mathbb{K})\subset \mathcal{Q}(S\mathbb{K})$,
one has $f_A\circ q_A=i({u_A})^*\circ m_A$.
By Lem.~\ref{mc1} and Lem.~\ref{mc2},
one has the following diagram with exact horizontal sequences
\[
\begin{tikzcd}
K_1(\tau_A(A)'\cap\mathcal{Q}(\mathbb{K})) \ar[r] \ar[d,"p_A"] & K_1(\mathcal{Q}(\mathbb{K})) \ar[d,"-I"] \ar[r,"\iota_A"] & Ext_s(A)\ar[d,"m_A"] \ar[r,"q_A"] & Ext_w(A) \ar[r] \ar[d,"f_A"] & 0 \ar[d,equal] \\
Ext(A, S\mathbb{K}) \ar[r,"u_A^*"] & Ext(\mathbb{C}, S\mathbb{K}) \ar[r,"e({u_A})^*"] & Ext (C_{u_A}, S\mathbb{K}) \ar[r,"i({u_A})^*"] & Ext (SA, S\mathbb{K}) \ar[r] & 0.
\end{tikzcd}
\]
We will show $m_A\circ \iota_A=e({u_A})^*\circ -I$. The statement then follows from the Five-Lemma.

Let $V\in\mathcal{M}(\mathbb{K})$ be the isometry with a rank 1 projection $e:=1-VV^*$.
It is enough to show $m_A\circ\iota_A([\pi(V)]_1)=e({u_A})^*\circ -I([\pi(V)]_1)$.
For a state $\psi$ of $A$ and $a(t)\in C_{u_A}$,
one has $ta(1)-\psi (a(t))\in C_0(0, 1)$.
Thus,
one has the following contractible completely positive lift of $\tau_{[e, 0]}\circ e({u_A}) : C_{u_A}\to \mathcal{Q}(S\mathbb{K})$:
\[\theta : C_{u_A}\ni a(t)\mapsto \psi (a(t))(1-VV^*)\in C[0, 1]\otimes\mathcal{M}(\mathbb{K})\subset\mathcal{M}(S\mathbb{K}).\]
A unital completely positive lift $L_{{\rm Ad}\pi(V)\circ \tau_A} : A\to \mathcal{M}(\mathbb{K})$ of ${\rm Ad}\pi(V)\circ\tau_A$ is given by
\[{\rm Ad}V\circ\rho+(1-VV^*)\psi,\]
and we have the following lift of $c({\rm Ad}\pi(V)\circ\tau_A) : C_{u_A}\to\mathcal{Q}(\mathbb{K})$:
\[\sigma : C_{u_A}\ni a(t)\mapsto V\rho (a(t))V^*+\psi(a(t))(1-VV^*)\in C[0, 1]\otimes\mathcal{M}(\mathbb{K})\subset\mathcal{M}(S\mathbb{K}).\]
Now one has a unitary
\[
U:=1_{C[0, 1]}\otimes
\left(\begin{array}{cc}
V&e\\
0&V^*
\end{array}\right)\in C[0, 1]\otimes\mathbb{M}_2\otimes\mathcal{M}(\mathbb{K})\subset\mathbb{M}_2\otimes\mathcal{M}(S\mathbb{K})
\]
satisfying ${\rm Ad} U\circ(({\rm id}_{C_0(0, 1]}\rho|_{C_{u_A}})\oplus\theta)=(\sigma\oplus 0)$.
Since $[c(\tau_A)]=0\in Ext(C_{u_A}, S\mathbb{K})$, this implies
\[e({u_A})^*\circ -I([\pi(V)]_1)=[\tau_{[e, 0]}\circ e({u_A})]=[c(\tau_A)\oplus\tau_{[e, 0]}\circ e({u_A})]=[c({\rm Ad}\pi(V)\circ\tau_A)\oplus 0]=m_A\circ\iota_A([\pi(V)]_1).\]
\end{proof}
Another way to understand
$Ext_s(A)$ is via the dual algebra $\mathfrak{D}(A)$ due to \cite{HR}. 
For the injective unital $*$-homomorphism $\rho : A\to\mathcal{M}(\mathbb{K})$ with $\rho(A)\cap\mathbb{K}=\{0\}$,
one has the dual algebra
\[\mathfrak{D}(A):=\{T\in\mathcal{M}(\mathbb{K})\;|\; [T, \rho(a)]\in\mathbb{K}\}.\]
For a projection $p\in\mathfrak{D}(A)\backslash \mathbb{K}$ (resp. $q\in \tau_A(A)'\cap\mathcal{Q}(\mathbb{K})$), there is an isometry $W\in \mathcal{M}(\mathbb{K})$  with $p=WW^*$(resp. $w\in\mathcal{Q}(\mathbb{K})$ with $q=ww^*$) which defines a unital essential extension ${\rm Ad}\pi(W^*)\circ\tau_A$ (resp. ${\rm Ad}w^*\circ\tau_A$). 
Since strong (weak) equivalence classes of the above extension do not depend on the choice of the isometry,
this gives the following isomorphisms (see \cite[{Chap. 5}]{HR})
\[K_0(\mathfrak{D}(A))\cong Ext_s(A),\quad K_0(\tau_A(A)'\cap\mathcal{Q}(\mathbb{K})))\cong Ext_w(A).\]
For $-[e]_0=[1_{\mathfrak{D}(A)}-e]_0\in K_0(\mathfrak{D}(A))$ and the isometry $V$ with $1-VV^*=e$,
$-[e]_0\in K_0(\mathfrak{D}(A))$ corresponds to $[{\rm Ad}\pi(V^*)\circ\tau_A]_s=\iota_A([\pi(V^*)]_1)\in Ext_s(A)$.
Thus,
one has the following commutative diagram
\[
\begin{tikzcd}
K_0(\mathbb{K}) \ar[r] & K_0(\mathfrak{D}(A)) \ar[d,"\cong"] \\
K_1(\mathcal{Q}(\mathbb{K})) \ar[u,"-{\rm Ind}"] \ar[r,"\iota_A"] & Ext_s(A).
\end{tikzcd}
\]
For the exact sequence
\[0\to\mathbb{K}\to\mathfrak{D}(A)\to \tau_A(A)'\cap\mathcal{Q}(\mathbb{K})\to 0,\]
the associated six-term exact sequence fits into
the following commutative diagram:
\[
\begin{tikzcd}[font=\footnotesize,every label/.append style = {font = \tiny},column sep=0.7cm]
K_1(\mathfrak{D}(A)) \ar[ddd] \ar[r] & K_1(\tau_A(A)'\cap\mathcal{Q}(\mathbb{K})) \ar[d,equal] \ar[r] & K_0(\mathbb{K})\ar[rr] && K_0(\mathfrak{D}(A))\ar[d,"\cong"] \ar[rr] && K_0(\tau_A(A)'\cap\mathcal{Q}(\mathbb{K})) \ar[d,"\cong"] \\
& K_1(\tau_A(A)'\cap\mathcal{Q}(\mathbb{K})) \ar[d,"p_A"] \ar[r] & K_1(\mathcal{Q}(\mathbb{K})) \ar[u,"-\rm Ind"] \ar[d,"-I"]\ar[rr,"\iota_A"] && Ext_s(A)\ar[d,"m_A"] \ar[rr,"q_A"] && Ext_w(A) \ar[d,"f_A"] \\
& Ext (A, S\mathbb{K}) \ar[r,"u_A^*"] \ar[d,"\eta_{A, \mathbb{C}}^{-1}"] & Ext (\mathbb{C}, S\mathbb{K})\ar[d,"\eta_{\mathbb{C}, \mathbb{C}}^{-1}"] \ar[rr,"e({u_A})^*"] && Ext(C_{u_A}, S\mathbb{K}) \ar[d,"\eta_{C_{u_A}, \mathbb{C}}^{-1}"] \ar[rr,"i({u_A})^*"] && Ext(SA, S\mathbb{K}) \ar[d,"\eta_{SA, \mathbb{C}}^{-1}"] \\
KK(SC_{u_A}, \mathbb{C}) \ar[r,"d({u_A})\hat{\otimes}"] & KK(A, \mathbb{C}) \ar[r,"KK(u_A)\hat{\otimes}"] & KK(\mathbb{C}, \mathbb{C})\ar[rr,"KK(e({u_A}))\hat{\otimes}"] && KK(C_{u_A}, \mathbb{C})\ar[rr,"KK(i({u_A}))\hat{\otimes}"] && KK(SA, \mathbb{C}).
\end{tikzcd}
\]
\begin{rem}\label{dD}
If $K_0(A), K_1(A)$ are finitely generated (i.e., $A$ has a Spanier--Whitehead K-dual $D(A)$),
$C_{u_A}$ is dualizable,
and one can observe the isomorphisms
\[K_*(\mathfrak{D}(A))\cong K_*(D(C_{u_A})),\quad K_*(\tau_A(A)'\cap\mathcal{Q}(\mathbb{K}))\cong K_{*-1}(D(A)).\]
\end{rem}
Summarising the results of this section,
we have the following corollary.
\begin{cor}\label{dDd}
Let $A$ be a separable nuclear unital C*-algebras, and let $W\in \mathcal{Q}(\mathbb{K})$ be a unitary with ${\rm Ind}(W)=1$.
There is an isomorphism $\Psi_A : Ext_s(A)\to KK(C_{u_A}, \mathbb{C})$ satisfying
\[I_S\otimes \Psi_A([E]_s)=KK(i(\pi_{C_{u_E}}))\hat{\otimes}KK(j(C_{u_E}))^{-1}\hat{\otimes}KK(e)^{-1}\in KK(SC_{u_A}, S),\]
\[\Psi_A(\iota_A([W]_1))=KK(e(u_A))\in KK(C_{u_A}, \mathbb{C})\]
for any unital essential extension $[\tau]_s=[E]_s$. 
\end{cor}
\begin{proof}
The isomorphism is defined by
\[\Psi_A :=-\eta_{C_{u_A}, \mathbb{C}}^{-1}\circ m_A.\]
Then, Lem. \ref{li} implies
\begin{align*}
I_S\otimes \Psi_A([E]_s)&=I_S\otimes (-\eta_{C_{u_A}, \mathbb{C}}^{-1}([c(\tau)]_s))\\
&=KK(i(\pi_{C_{u_E}}))\hat{\otimes}KK(j(C_{u_E}))^{-1}\hat{\otimes}KK(e)^{-1},
\end{align*}
and Thm. \ref{mcs} implies
\begin{align*}
-\Psi_A(\iota_A([W]_1))&=-\eta_{C_{u_A}, \mathbb{C}}^{-1}(m_A(\iota_A(-[W]_1)))\\
&=-\eta_{C_{u_A}, \mathbb{C}}^{-1}(e(u_A)^*([\tau_{[e, 0]}]))\\
&=-\eta_{C_{u_A}, \mathbb{C}}^{-1}([\tau_{[e\circ e(u_A), 0]}])\\
&=-[e\circ e(u_A), 0]\\
&=-[e(u_A)\otimes e, 0]\\
&=-KK(e(u_A)). \qedhere
\end{align*}
\end{proof}
\section{Strong K-theoretic duality for unital extensions}\label{skD}
We first recall the definition of this duality from \cite{M}:
\begin{dfn}[{\cite[{Def. 6.1.}]{M}}]\label{skd}
Let $A, B$ be separable nuclear unital C*-algebras,
and let $\tau : A\to \mathcal{Q}(\mathbb{K})$ and $\sigma : B\to\mathcal{Q}(\mathbb{K})$ be unital essential extensions with $E:=\pi^{-1}(\tau(A)), F:=\pi^{-1}(\sigma(B))$.
Two extension $\tau, \sigma$ are K-theoretic dual if there are vertical arrows given by isomorphisms that make the following diagram commute:
\[
\begin{tikzcd}
K_1(\mathfrak{D}(A)) \ar[r] \ar[d] & K_1(\tau_A(A)'\cap\mathcal{Q}(\mathbb{K})) \ar[d] \ar[r] & K_1(\mathcal{Q}(\mathbb{K})) \ar[d,"\rm Ind"] \ar[r,"\iota_A"] & Ext_s(A) \ar[d,"\Phi_A"] \ar[r,"q_A"] & Ext_w(A) \ar[d] \\
K_1(F) \ar[r] & K_1(B) \ar[r,"\rm Ind"] & K_0(\mathbb{K}) \ar[r] & K_0(F) \ar[r,"K_0(\pi_F)"] & K_0(B),
\end{tikzcd}
\]
\[
\begin{tikzcd}
K_1(\mathfrak{D}(B)) \ar[r] \ar[d] & K_1(\tau_B(B)'\cap\mathcal{Q}(\mathbb{K})) \ar[d] \ar[r] & K_1(\mathcal{Q}(\mathbb{K})) \ar[d,"\rm Ind"] \ar[r,"\iota_B"] & Ext_s(B) \ar[d,"\Phi_B"] \ar[r,"q_B"] & Ext_w(B) \ar[d] \\
K_1(E) \ar[r] & K_1(A) \ar[r,"\rm Ind"] & K_0(\mathbb{K}) \ar[r] & K_0(E) \ar[r,"K_0(\pi_E)"] & K_0(A).
\end{tikzcd}
\]
We call $\tau$ and $\sigma$ are strongly K-theoretic dual with respect to $\epsilon\in \{\pm 1\}$ if
\[\Phi_A([E]_s)=\epsilon [1_F]_0,\quad \Phi_B([F]_s)=\epsilon [1_E]_0\]
holds.
\end{dfn}
\begin{rem}
It is easy to see that $E$ and $F$ are strongly K-theoretic dual with respect to $\epsilon$ if and only if we have the following isomorphisms
\[\Phi_B : (K_0(E), \epsilon [1_E]_0, [e]_0, K_1(E))\cong (Ext_s(B), [F]_s, \iota_B([W]_1), K_1(\mathfrak{D}(B))),\]
\[\Phi_A : (K_0(F), \epsilon [1_F]_0, [e]_0, K_1(F))\cong (Ext_s(A), [E]_s, \iota_A([W]_1), K_1(\mathfrak{D}(A))),\] 
where $W\in \mathcal{Q}(\mathbb{K})$ is a unitary of ${\rm Ind}(W)=1$.
\end{rem}
\begin{rem}\label{r1}
Thanks to \cite[{Thm. A}]{GR},
if both of $A$ and $B$ are Kirchberg algebras,
then the isomorphism class of $F$ is uniquely determined by $E$ and vice versa.
\end{rem}

\begin{rem}\label{md}
In \cite{M},
K. Matsumoto computed the strong extension groups of the Cuntz--Krieger algebras explicitly and discovered the isomorphism
\[(K_0(\mathcal{T}_{A^t}), -[1_{\mathcal{T}_{A^t}}]_0, [e]_0, K_1(\mathcal{T}_{A^t}))\cong (Ext_s(\mathcal{O}_A), [\mathcal{T}_A]_s, \iota_A([W]_1), K_1(\mathfrak{D}(\mathcal{O}_A)),\] 
where $\mathcal{T}_A$ is the Toeplitz extension of $\mathcal{O}_A$ (see \cite{EFW, DE, M0}).
In particular, the following Toeplitz extensions are proved to be strongly K-theoretic dual with respect to $\epsilon=-1$:
\[\mathbb{K}\to\mathcal{T}_A\to\mathcal{O}_A,\quad \mathbb{K}\to\mathcal{T}_{A^t}\to\mathcal{O}_{A^t}.\]
\end{rem}

Let $\xi_E$ (resp. $\xi_F$) be the inclusion $\mathbb{K}+\mathbb{C}1_E\to E$ (resp. $\mathbb{K}+\mathbb{C}1_F\to F$) and denote by $q_E$ (resp.~$q_F$) the quotient map $C_{\xi_E}\to C_{\xi_E}/C_0(0, 1]\otimes\mathbb{K}=C_{u_A}$ (resp. $C_{\xi_F}\to C_{u_B}$).
Let $i_\mathbb{K}$ (resp. $i_\mathbb{C}$) be the inclusion $\mathbb{K}\hookrightarrow\mathbb{K}+\mathbb{C}$ (resp. $\mathbb{C}\hookrightarrow \mathbb{K}+\mathbb{C}$).
For the elements $\Psi_B([F]_s), KK(e(u_B))$ (see Cor.~\ref{dDd}),
we write
\[\overline{\Psi_B([F]_s)}:=\Psi_B([F]_s)\hat{\otimes}KK(e)\hat{\otimes}KK(i_\mathbb{K})\in KK(C_{u_B}, \mathbb{K}+\mathbb{C}),\]
\[\overline{KK(e(u_B))}:=KK(e(u_B))\hat{\otimes}KK(i_\mathbb{C})\in KK(C_{u_B}, \mathbb{K}+\mathbb{C}).\]

Using Cor. \ref{Sed} and the following two propositions which will be proved in appendix,
we give a categorical picture to understand strong K-theoretic duality and prove the existence of dual extensions (Thm. \ref{mtcp}).

\begin{prop}\label{keyp2}
The quotient map $q_F$ gives a KK-equivalence $KK(q_F)\in KK(C_{\xi_F}, C_{u_B})^{-1}$ and one has
\[KK(q_F)\hat{\otimes}(\overline{\Psi_B([F]_s)}+\overline{KK(e(u_B))})=KK(e(\xi_F))\in KK(C_{\xi_F}, \mathbb{K}+\mathbb{C}).\]
\end{prop}

\begin{prop}\label{keyp1}
Let $A, B$ and $E, F$ be as in Def. \ref{skd}, and assume that they satisfy UCT.
Then, the isomorphism
\[(K_0(E), \epsilon [1_E]_0, \delta [e]_0, K_1(E))\cong (Ext_s(B), [F]_s, \iota_B([W]_1), K_1(\mathfrak{D}(B))),\]
holds if and only if
there exists a duality class $\mu_E\in KK(\mathbb{C}, C_{\xi_F}\otimes E)$ satisfying
\[D_{\mu_E, \nu_{\epsilon, \delta}}(KK(e(\xi_F)))=KK(\xi_E)\]
\end{prop}

\begin{thm}\label{mtcp}
Let $A$ be a unital separable nuclear UCT C*-algebras with finitely generated K-groups,
and let $\mathbb{K}\to E\to A$ be a unital essential extension.
Then the following holds:
\begin{enumerate}
\item{} There exists a unital separable nuclear UCT C*-algebra $B$ and a unital essential extension $\mathbb{K}\to F\to B$ which is strongly K-theoretic dual to $\mathbb{K}\to E\to A$ with respect to $\epsilon\in\{\pm 1\}$.
\item{} Two extensions $E$ and $F$ are strongly K-theoretic dual if and only if there exist duality classes
\[\mu_1\in KK(\mathbb{C}, C_{\xi_F}\otimes E),\quad \nu_2\in KK(C_{\xi_E}\otimes F, \mathbb{C})\]
making the following diagram commute
\[
\begin{tikzcd}[column sep=3.4cm]
E \ar[d,equal] & \ar[l,"D_{\mu_1, \nu_{\epsilon, +1}}(KK(e(\xi_F)))"] \mathbb{C}+\mathbb{K} \ar[d,equal] & \ar[l,"D_{\mu_{\epsilon, +1}, \nu_2}(KK(\xi_F))"] C_{\xi_E} \ar[d,equal] \\
E & \ar[l,"KK(\xi_E)"] \mathbb{C}+\mathbb{K} & \ar[l,"KK(e(\xi_E))"] C_{\xi_E}
\end{tikzcd}
\]
(i.e., a dual sequence of $C_{\xi_F}\xrightarrow{e(\xi_F)}\mathbb{K}+\mathbb{C}\xrightarrow{\xi_F}F$ is given by $E\xleftarrow{\xi_E}\mathbb{C}+\mathbb{K}\xleftarrow{e(\xi_E)}C_{\xi_E}$.).
\end{enumerate}
\end{thm}
\begin{lem}\label{exB}
Let $A, E$ be as in Thm. \ref{mtcp}, and let $W\in\mathcal{Q}(\mathbb{K})$ be a unitary of ${\rm Ind}(W)=1$.
Then, there (uniquely) exists a unital UCT Kirchberg algebra $B$ satisfying
\[(K_0(E), [e]_0, K_1(E))\cong (Ext_s(B), \iota_B([W]_1), K_1(\mathfrak{D}(B))).\]
\end{lem}
\begin{proof}
Let $R$ be the unital UCT Kirchberg algebra defined by
\[(K_0(R), [1_R]_0, K_1(R))\cong (K_0(E), [e]_0, K_1(E)).\]
By \cite[{Thm. 3.3.}]{S},
there exist a unital Kirchberg algebra $B$ reciprocal to $R$ and a duality class $\mu\in KK(\mathbb{C}, C_{u_B}\otimes R)$ satisfying $KK(u_R)=\mu\hat{\otimes}(KK(e(u_B))\otimes I_R)$.
The UCT gives a KK-equivalence $\gamma\in KK(R, E)^{-1}$ satisfying $KK(u_R)\hat{\otimes}\gamma=[e]_0\in KK(\mathbb{C}, E)=K_0(E)$,
and the isomorphism
\[Ext_s(B)\xrightarrow{\Psi_B}KK(C_{u_B}, \mathbb{C})\xrightarrow{\mu\hat{\otimes}}KK(\mathbb{C}, R)\xrightarrow{\hat{\otimes}\gamma}KK(\mathbb{C},  E)=K_0(E)\]
sends $\iota_B([W]_1)$ to $[e]_0$ by Cor. \ref{dDd}.
The reciprocality (i.e,. $D(C_{u_B})=R$) and Rem. \ref{dD} imply $$K_1(\mathfrak{D}(B))\cong K_1(D(C_{u_B}))\cong K_1(R)\cong K_1(E)$$
and this completes the proof.
\end{proof}
\begin{proof}[{The proof of Thm. \ref{mtcp}}]
Since statement 2 immediately follows from Prop. \ref{keyp1} and Lem. \ref{inv},
we only have to show statement 1.
By Lem. \ref{exB},
one has a unital UCT Kirchberg algebra $B$ with the following isomorphism
\[\Phi_B : (K_0(E), [e]_0, K_1(E))\cong (Ext_s(B), \iota_B([W]_1), K_1(\mathfrak{D}(B))),\]
and there is a unital essential extension $F$ defined by 
\[\Phi_B(\epsilon [1_E]_0)=[F]_s.\]
By Prop. \ref{keyp1},
one has a duality class $\mu_E\in KK(\mathbb{C}, C_{\xi_F}\otimes E)$ satisfying
\[D_{\mu_E, \nu_{\epsilon, +1}}(KK(e(\xi_F)))=KK(\xi_E).\]
Since $KK(\sigma_{\mathbb{K}+\mathbb{C}, \mathbb{K}+\mathbb{C}})\hat{\otimes}\nu_{\epsilon, +1}=\nu_{+1, \epsilon}$,
Cor.~\ref{Sed} shows that
there exists a duality classes $\mu_F\in KK(\mathbb{C}, C_{\xi_E}\otimes F)$ satisfying
\[D_{\mu_F, \nu_{+1, \epsilon}}(KK(e(\xi_E)))=KK(\xi_F).\]
Now Prop. \ref{keyp1} gives an isomorphism
\[\epsilon\Phi_A : (K_0(F), [1_F]_0, \epsilon [e]_0, K_1(F))\cong (Ext_s(A), [E]_s, \iota_A([W]_1), K_1(\mathfrak{D}(A))). \qedhere\]
\end{proof}
\begin{ex}
For the Cuntz algebra $\mathcal{O}_n$,
one has $$(Ext_s(\mathcal{O}_n), \iota_{\mathcal{O}_n}([W]_1), [E(1)]_s)\cong (\mathbb{Z}, n-1, 1),$$ where $W\in U(\mathcal{Q}(\mathbb{K}))$ is a unitary with ${\rm Ind} ([W]_1)=1$ and $E(1):=E_n$ is the Cuntz--Toeplitz algebra.
Denote by $E(m)$ the extension which satisfies $[E(m)]_s=m[E(1)]_s$.
For $m\in \mathbb{N}$,
the algebra $E(m)$ is given by
\[E(m):=\mathbb{M}_m(\mathbb{C})\otimes\mathbb{K}+1_m\otimes E(1)\subset \mathbb{M}_m(\mathbb{C})\otimes \mathcal{M}(\mathbb{K}).\]
Since $$(K_0(\mathbb{M}_m(E(-\epsilon))), [e]_0, [1]_0)\cong (\mathbb{Z}, n-1, \epsilon m),$$ 
Rem. \ref{r1} implies that the strong K-theoritic dual of $$\mathbb{K}\to E(m)\to \mathcal{O}_n, \quad m>0$$ with respect to $\epsilon \in \{\pm 1\}$ is $$\mathbb{M}_m(\mathbb{K})\to \mathbb{M}_m(E(-\epsilon))\to\mathbb{M}_m(\mathcal{O}_n).$$
\end{ex}


\section{Appendix}
\subsection{Proof of Prop. \ref{keyp2}}
Denote by $p_\mathbb{K}$ and $p_\mathbb{C}$ the morphisms defined by 
\[p_\mathbb{K} \in KK(\mathbb{K}+\mathbb{C}, \mathbb{K})=\operatorname{Hom}(K_0(\mathbb{K}+\mathbb{C}), K_0(\mathbb{K})),\quad p_\mathbb{K}\hat{\otimes}[e]_0=[e]_0,\; p_\mathbb{K}\hat{\otimes}[1]_0=0,\]
\[p_\mathbb{C}\in KK(\mathbb{K}+\mathbb{C}, \mathbb{C})=\operatorname{Hom}(K_0(\mathbb{K}+\mathbb{C}), K_0(\mathbb{C})),\quad p_\mathbb{C}\hat{\otimes}[e]_0=0,\; p_\mathbb{C}\hat{\otimes}[1]_0=[1]_0.\]
One has $p_\mathbb{K}\hat{\otimes}KK(i_\mathbb{K})+p_\mathbb{C}\hat{\otimes}KK(i_\mathbb{C})=I_{\mathbb{K}+\mathbb{C}}$.
\begin{lem}\label{k2}
The following diagram commutes:
\[
\begin{tikzcd}[column sep=2.4cm]
SC_{\xi_F} \ar[r,"I_S\otimes KK(e({\xi_F}))"] \ar[d,"I_S\otimes KK(q_F)" left] & S(\mathbb{K}+\mathbb{C}1) \ar[d,"I_S\otimes p_\mathbb{K}"] \\
SC_{u_B} \ar[r,"I_S\otimes (\Psi_B({[F]}_s)\hat{\otimes}KK(e))"] & S\mathbb{K}.
\end{tikzcd}
\]
In particular, one has $KK(e(\xi_F))\hat{\otimes}p_\mathbb{K}\hat{\otimes}KK(i_\mathbb{K})=KK(q_F)\hat{\otimes}\overline{\Psi_B([F]_s)}$
\end{lem}
\begin{proof}
Recall the Busby invariant $\sigma : B\to\mathcal{Q}(\mathbb{K})$ of the extension $\mathbb{K}\to F\to B$.
We write $CB:=C_0(0, 1]\otimes B, CF:=C_0(0, 1]\otimes F$ for short.
We use the following identification
\[SC_{\xi_F}=\{f_t(s)\in C_0(0, 1]\otimes (CF)\; |\;f_t(1)\in \mathbb{K}+\mathbb{C}1_F,\; f_1(s)=0 \},\]
\[SC_{u_B}=\{b_t(s)\in C_0(0, 1]\otimes (CB)\; |\; b_t(1)\in \mathbb{C}1_B,\; b_1(s)=0\}, \]
\begin{align*}
&C_{\pi_{C_{u_F}}}=\\
&\{(b_t(s), f(s))\in (C_0(0, 1]\otimes(CB))\oplus (CF)\;|\; b_1(s)=\pi_F(f(s)),\; b_t(1)\in \mathbb{C}1_B,\; f(1)\in\mathbb{C}1_F\},
\end{align*}
where the third algebra is the mapping cone obtained from the extension 
\[[c(\sigma)]=[S\mathbb{K}\to C_{u_F}\xrightarrow{\pi_{C_{u_F}}} C_{u_B}].\]

For the mapping cone $C_{i_\mathbb{C}}=\{f(s)\in C_0(0, 1]\otimes (\mathbb{K}+\mathbb{C})\;|\; f(1)\in\mathbb{C}\}$,
we define two maps $x : S\mathbb{K}\to C_{i_\mathbb{C}}$ and $y : C_{i_\mathbb{C}}\to C_{\pi_{C_{u_F}}}$ by
\[x : S\mathbb{K}\ni f(s)\mapsto f(s)\in C_{i_\mathbb{C}},\]
\[y : C_{i_\mathbb{C}}\ni f(s)\mapsto (\pi_{F}(f(ts)), f(s))\in C_{\pi_{C_{u_F}}},\]
where $\pi_F : F\to F/\mathbb{K}=B$ is the quotient map.
Consider the following diagram:
\[
\begin{tikzcd}
SC_{\xi_F} \ar[dd,"Sq_F"] \ar[r,"Se({\xi_F})"] & S(\mathbb{K}+\mathbb{C}1) \ar[d,"i(i_\mathbb{C})" left] & & \\
& C_{i_\mathbb{C}} \ar[d,"y" left] & S\mathbb{K} \ar[l,"x"] \ar[ul] \ar[dl,"j{(C_{u_F})}"] \\
SC_{u_B} \ar[r,"i{(\pi_{C_{u_F}})}"] & C_{\pi_{C_{u_F}}} &,
\end{tikzcd}
\]
where the triangles on the right hand side commute.
By the extension
\[0\to S\mathbb{K}\xrightarrow{x} C_{i_\mathbb{C}}\to C_0(0, 1]\to 0,\]
the inclusion $x$ is a KK-equivalence and $y=x^{-1}\hat{\otimes}j(C_{u_F})$ is also a KK-equivalence.
The exact sequence
\[K_*(S)\xrightarrow{K_*(Si_\mathbb{C})} K_*(S(\mathbb{K}+\mathbb{C}1))\xrightarrow{K_*(i(i_{\mathbb{C}}))} K_*(C_{i_\mathbb{C}}),\]
and the commutativity of the upper right triangle implies $KK(i(i_\mathbb{C}))\hat{\otimes}KK(x)^{-1}=I_S\otimes p_{\mathbb{K}}$.

Since Cor. \ref{dDd} shows $KK(i{(\pi_{C_{u_F}}}))\hat{\otimes}KK(j{(C_{u_F})})^{-1}=I_S\otimes(\Psi_B([F]_s)\hat{\otimes}KK(e))$,
it is enough to check that the large square commutes up to homotopy.
For $\varphi : =y\circ i(i_\mathbb{C})\circ Se({\xi_F})$ and $\psi : =i{(\pi_{C_{u_F}})}\circ Sq_F$,
one has
\[\varphi : f_t(s)\mapsto (\pi_F(f_{ts}(1)), f_s(1)),\quad \psi : f_t(s)\mapsto (\pi_F(f_t(s)), 0).\]
It is straightforward to check that the following maps from $SC_{\xi_F}$ to $C_{\pi_{C_{u_F}}}$ are well-defined for $h\in[0, 1]$:
\[\varphi_h : f_t(s)\mapsto (\pi_F(f_{ts}(hs+(1-h))), f_s(hs+(1-h))),\]
\[\psi_h : f_t(s)\mapsto (\pi_F(f_{t(hs+(1-h))}(s)), f_{(hs+(1-h))}(s)),\]
and one has 
\[\varphi=\varphi_0\sim_h\varphi_1=\psi_1\sim_h\psi_0=\psi. \qedhere\]
\end{proof}
\begin{proof}[{Proof of Prop. \ref{keyp2}}]
The kernel of $q_F$ is $C_0(0, 1]\otimes\mathbb{K}$ and this implies $q_F$ is a KK-equivalence.

We have $\overline{KK(e(u_B))}=KK(e(u_B))\hat{\otimes}KK(i_\mathbb{C})$ by definition,
and it is easy to check $\pi\circ e(\xi_F)=e(u_B)\circ q_F$ for the quotient map $\pi : \mathbb{K}+\mathbb{C}\to \mathbb{C}$.
Since $p_\mathbb{C}=KK(\pi)$,
one has
\[KK(e(\xi_F))\hat{\otimes}p_\mathbb{C}\hat{\otimes}KK(i_\mathbb{C})=KK(q_F)\hat{\otimes}\overline{KK(e(u_B))},\]
and Lem. \ref{k2} shows
\[KK(e(\xi_F))=KK(e(\xi_F))\hat{\otimes}(p_\mathbb{K}\hat{\otimes}KK(i_\mathbb{K})+p_\mathbb{C}\hat{\otimes}KK(i_\mathbb{C}))=KK(q_F)\hat{\otimes}(\overline{\Psi_B([F]_s)}+\overline{KK(e(u_B))}).\]
\end{proof}

\subsection{Proof of Prop. \ref{keyp1}}
We use the following lemma proved in \cite{KPW}.
\begin{lem}\cite[{Sec. 4.4.}]{KPW}\label{fg}
Let $A, B, E, F$ be as in Prop. \ref{keyp1} (i.e., they are separable nuclear UCT C*-algebras.).
If there is an isomorphism
\[\Phi_B : (K_0(E), \epsilon [1_E]_0, \delta [e]_0, K_1(E))\cong (Ext_s(B), [F]_s, \iota_B([W]_1), K_1(\mathfrak{D}(B))),\]
the K-groups $K_i(C_{u_B}), \;i=0, 1$ are finitely generated.
In particular, the K-groups of $E, A, B, F$ are finitely generated.
\end{lem}
\begin{proof}
The separability of $E$ implies that the groups $Ext_s(B)\cong KK(C_{u_B}, \mathbb{C}),\; K_1(\mathfrak{D}(B))\cong KK^1(C_{u_B}, \mathbb{C})$ are countable.
Thus, UCT implies that $\operatorname{Hom}(K_i(C_{u_B}), \mathbb{Z})$ and $\operatorname{Ext}^1_\mathbb{Z}(K_i(C_{u_B}), \mathbb{Z})$ are countable groups.
Now the same argument as in \cite[{Sec. 4.4.}]{KPW} shows the statement.
\end{proof}

\begin{proof}[{Proof of Prop. \ref{keyp1}}]
First, we prove the if-direction.
For a duality class $\mu_E\in KK(\mathbb{C}, C_{\xi_F}\otimes E)$,
we define a duality class $\mu$ by
\[\mu:=\mu_E\hat{\otimes}(KK(q_F)\otimes I_E)\in KK(\mathbb{C}, C_{u_B}\otimes E).\]
Prop.~\ref{keyp2} and the assumption imply
\[D_{\mu, \nu_{\epsilon, \delta}}(\overline{\Psi_B([F]_s)}+\overline{KK(e(u_B))})=KK(\xi_E).\]
Since $(KK(i_\mathbb{C})\otimes I_{\mathbb{K}+\mathbb{C}})\hat{\otimes}\nu_{\epsilon, \delta}=\epsilon p_\mathbb{K}\hat{\otimes}KK(e)^{-1}\in KK(\mathbb{K}+\mathbb{C}, \mathbb{C})$,
one has
\begin{align*}
[1_E]_0=&KK(i_\mathbb{C})\hat{\otimes}D_{\mu, \nu_{\epsilon, \delta}}(\overline{\Psi_B([F]_s)}+\overline{KK(e(u_B))})\\
=&\mu\hat{\otimes}((\overline{\Psi_B([F]_s)}+\overline{KK(e(u_B))})\otimes I_E)\hat{\otimes}((\epsilon p_\mathbb{K}\hat{\otimes}KK(e)^{-1})\otimes I_E)\\
=&\mu\hat{\otimes}(\epsilon \Psi_B([F]_s)\otimes I_E).
\end{align*}
Similarly, the equation $(KK(e)\otimes I_{\mathbb{K}+\mathbb{C}})\hat{\otimes}\nu_{\epsilon, \delta}=\delta p_\mathbb{C}\in KK(\mathbb{K}+\mathbb{C}, \mathbb{C})$ implies
\begin{align*}
[e]_0=&KK(e)\hat{\otimes}D_{\mu, \nu_{\epsilon, \delta}}(\overline{\Psi_B([F]_s)}+\overline{KK(e(u_B))})\\
=&\mu\hat{\otimes}((\overline{\Psi_B([F]_s)}+\overline{KK(e(u_B))})\otimes I_E)\hat{\otimes}(\delta p_\mathbb{C}\otimes I_E)\\
=&\mu\hat{\otimes}(\delta KK(e(u_B))\otimes I_E).
\end{align*}
Now we have the desired isomorphisms
\[Ext_s(B)\xrightarrow{\Psi_B} KK(C_{u_B}, \mathbb{C})\xrightarrow{\mu\hat{\otimes}(-\otimes I_E)}KK(\mathbb{C}, E)=K_0(E),\]
\[K_1(\mathfrak{D}(B))\cong KK(C_{u_B}, S)\xrightarrow{\mu\hat{\otimes}}KK(\mathbb{C}, SE)\cong K_1(E).\]

Next, we show the only if-direction.
We identify $[1_E]_0, [e]_0\in K_0(E)$ with $KK(u_E), KK(\mathbb{C}\to \mathbb{C}e\subset E)\in KK(\mathbb{C}, E)$.
Assume that there is an isomorphism
\[\Phi_B : (K_0(E), \epsilon [1_E]_0, \delta [e]_0, K_1(E))\to (Ext_s(B), [F]_s, \iota_B([W]_1), K_1(\mathfrak{D}(B))).\]

Lem. \ref{fg}, Thm. \ref{ks} and the UCT imply that $E$ and $C_{u_B}$ are Spanier--Whitehead K-dual with a duality class $\mu\in KK(\mathbb{C}, C_{u_B}\otimes E)$.
The UCT gives a KK-equivalence $\gamma \in KK(E, E)^{-1}$ making the following diagram commute
\[
\begin{tikzcd}[column sep=2.5cm]
KK(C_{u_B}, \mathbb{K}) \ar[d,"\Phi_B^{-1}\circ \Psi_B^{-1}"] \ar[r,"\mu\hat{\otimes}"] & KK(\mathbb{C}, E)\ar[d,dashed,"\hat{\otimes}\gamma"] \\
KK(\mathbb{C}, E)\ar[r,"{\rm id}"] & KK(\mathbb{C},  E).
\end{tikzcd}
\]
We show that \[\mu_E:=\mu\hat{\otimes}(KK(q_F)^{-1}\otimes I_E)\hat{\otimes}(I_{C_{\xi_F}}\otimes \gamma)\in KK(\mathbb{C}, C_{\xi_F}\otimes E)\]
satisfies $D_{\mu_E, \nu_{\epsilon, \delta}}(KK(e(\xi_F)))=KK(\xi_E)$.
Similar computation as in the  if part yields
\begin{align*}
&KK(i_\mathbb{K})\hat{\otimes}D_{\mu_E, \nu_{\epsilon, \delta}}(KK(e(\xi_F)))\\
=&KK(e)^{-1}\hat{\otimes}\mu\hat{\otimes}(\delta KK(e(u_B))\otimes I_E)\hat{\otimes}\gamma\\
=&KK(e)^{-1}\hat{\otimes}\Phi_B^{-1}\circ\Psi_B^{-1}(\delta KK(e(u_B)))\\
=&KK(e)^{-1}\hat{\otimes}\delta^2 KK(\mathbb{C}\to \mathbb{C}e\subset E)\\
=&KK(\mathbb{K}\hookrightarrow E)(=[e]_0)
\end{align*}
and
\begin{align*}
&KK(i_\mathbb{C})\hat{\otimes}D_{\mu_E, \nu_{\epsilon, \delta}}(KK(e(\xi_F)))\\
=&\mu\hat{\otimes}(\epsilon \Psi_B([F]_s\otimes I_E)\hat{\otimes}\gamma\\
=&\Phi_B^{-1}\circ\Psi_B^{-1}(\epsilon \Psi_B([F]_s))\\
=&KK(u_E)(=[1_E]_0).
\end{align*}
Thus, the UCT $$KK(\mathbb{K}+\mathbb{C}, E)=\operatorname{Hom}(K_0(\mathbb{K}+\mathbb{C}), K_0(E))$$ implies
\begin{align*}
D_{\mu_E, \nu_{\epsilon, \delta}}(KK(e(\xi_F)))=&(p_\mathbb{K}\hat{\otimes}KK(i_\mathbb{K})+p_\mathbb{C}\hat{\otimes}KK(i_\mathbb{C}))\hat{\otimes}D_{\mu_E, \nu_{\epsilon, \delta}}(KK(e(\xi_F)))\\
=&p_\mathbb{K}\hat{\otimes}KK(\mathbb{K}\hookrightarrow E)+p_\mathbb{C}\hat{\otimes}KK(\mathbb{C}\xrightarrow{u_E}E)\\
=&KK(\mathbb{K}+\mathbb{C}\xrightarrow{\xi_E} E). \qedhere
\end{align*}
\end{proof}

\subsection{The isomorphism $\eta_{A,B} \colon KK(A,B) \to Ext(A, SB \otimes \mathbb{K})$ in \cref{eqn:Blackadar-19.2.6}}
In this section we will give a proof of the fact that the map $\eta_{A,B}$ defined in \cref{eqn:Blackadar-19.2.6,eqn:Blackadar-def} is an isomorphism for two nuclear $C^*$-algebras $A$ and $B$, since we could not find a proof of this statement in the literature. We will show that $\eta_{A,B}$ fits into the following commutative diagram
\[
	\begin{tikzcd}
		KK(A,B) \ar[r,"\eta_{A,B}"] \arrow[d,"b" left, "\cong" right] & Ext(A,SB \otimes \mathbb{K}) \\
		KK^1(A,SB) \arrow[ur, "\cong" below]
	\end{tikzcd}
\]
The map $b \colon KK(A,B) \to KK^1(A, SB)$ is given by Bott periodicity and therefore an isomorphism \cite[Cor.~19.2.2]{B}. In the Cuntz picture for the group $KK^1(A,SB)$ a class $[\psi, P]$ is represented by a projection $P \in M(SB\otimes\mathbb{K})$ and a $*$-homomorphism $\psi \colon A \to M(SB\otimes\mathbb{K})$ with $P\psi(a) - \psi(a)P \in SB \otimes \mathbb{K}$ for all $a \in A$. The diagonal map sends $[\psi,P]$ to the Busby invariant 
\[
	\tau_{[\psi,P]}(a) = \pi(P\psi(a)P)\ .
\]
This provides an isomorphism by \cite[Prop.~17.6.5]{B}. Hence $\eta_{A,B}$ will turn out to be an isomorphism, once we have shown that the above diagram commutes. 

The homomorphism $b$ is given by the Kasparov product with a class $\mathbf{b} \in KK(\C, C_0(\R, \C l_1))$ given by the Kasparov module 
\[	
	\left(\lambda,C_0(\R,\C l_1), F \right)\ ,
\] 
where $\lambda \colon \C \to M(C_0(\R,\C l_1))$ is the unit homomorphism and $F$ is the multiplier on $C_0(\R, \C l_1)$ corresponding to the function $xg\,(1 + x^2)^{-\tfrac{1}{2}}$, where $g \in \C l_1$ is the generator with $g^2 = 1$. Note that $x \mapsto x\,(1 + x^2)^{-\tfrac{1}{2}}$ provides a homeomorphism $\R \to (-1,1)$ with inverse map $y \mapsto y\,(1 - x^2)^{-\tfrac{1}{2}}$, which induces a $*$-isomorphism 
\[
	\theta \colon C_0((-1,1), \C l_1) \to C_0(\R,\C l_1)\ .
\]
Pulling back $\mathbf{b}$ with $\theta$ turns it into the Kasparov module
\[	
	(\lambda, C_0((-1,1), \C l_1), \hat{F})\ ,
\]
where $\hat{F}$ is the multiplier on $C_0((-1,1), \C l_1)$ corresponding to the function $x \mapsto x\,g$. On the interval $[-1,1]$ the identity function is homotopic relative to its endpoints to $s(x) = \sin(\tfrac{\pi}{2}x)$. Hence, we may replace $\hat{F}$ by the multiplier $\widetilde{F}$ corresponding to $s \cdot g$ without changing the $KK$-class. For the rest of this section we will identify $S$ with $C_0(-1,1)$. The Bott class is then represented by the Kasparov module
\[
	(\lambda, S\C l_1, \widetilde{F}) \in KK(\C, S\C l_1)\ .
\]
Let $[\phi_0, \phi_1] \in KK(A,B)$ be a Cuntz pair. Let $H_B = \ell^2(\N) \otimes B$ and $\hat{H}_B = H_B^{(0)} \oplus H_B^{(1)}$ with the superscripts denoting the even, respectively odd part. The class $[\phi_0, \phi_1]$ corresponds to the Kasparov module
\[
	\left(\phi_0 \oplus \phi_1, \hat{H}_B, \begin{pmatrix} 0 & 1 \\ 1 & 0 \end{pmatrix}\right)\ .
\]
The internal graded tensor product $\hat{H}_B \otimes_B (S\C l_1 \otimes B)$ is isomorphic to the external graded tensor product $\hat{H}_{SB} \otimes \C l_1$. Note that the adjointable $S\C l_1 \otimes B$-linear maps on this module are isomorphic to $\hat{M}_2(\C) \otimes M(SB \otimes \bK) \otimes \C l_1$. This is a graded tensor product of $C^*$-algebras, where $\hat{M}_2(\C)$ denotes the complex $2 \times 2$-matrices with the diagonal/off-diagonal grading. 

Let $c(x) = \cos(\tfrac{\pi}{2}x)$ and observe that $c \in S$. 
\begin{lem}
	Let $(\phi_0, \phi_1)$ be a Cuntz pair representing a class in $KK(A,B)$. The Kasparov intersection product $[\phi_0, \phi_1] \widehat{\otimes} (\mathbf{b} \otimes I_B) \in KK^1(A,SB)$ is represented by the Kasparov module
	\begin{equation} \label{eqn:Kasparov_product}
		\left((\phi_0 \oplus \phi_1) \otimes 1, \hat{H}_{SB} \otimes \C l_1, G \right) \ ,		
	\end{equation}
	where $G \in \hat{M}_2(\C) \otimes M(SB \otimes \bK) \otimes \C l_1$ is the odd operator
	\[
		G = \begin{pmatrix}
			s & 0 \\
			0 & s 
		\end{pmatrix} \otimes g + 
		\begin{pmatrix}
			0 & c \\
			c & 0 
		\end{pmatrix} \otimes 1
	\]
\end{lem}

\begin{proof}
For $\xi \in \hat{H}_{B}$ let $M_{\xi} \colon S\C l_1 \to \hat{H}_{SB} \otimes \C l_1$ be the module map that sends $f \otimes x \in S\C l_1$ to $f\xi \otimes x \in \hat{H}_{SB} \otimes \C l_1$. Let $\xi^{(i)} \in H_B^{(i)}$. Then we have
\begin{align*}
	& G \cdot M_{\xi^{(0)} \oplus 0}(f \otimes x) - M_{\xi^{(0)} \oplus 0}((s \otimes g) \cdot (f \otimes x)) =  (0 \oplus cf\xi^{(0)}) \otimes x\ ,\\
	& G \cdot M_{0 \oplus \xi^{(1)}}(f \otimes x) + M_{0 \oplus \xi^{(1)}}((s \otimes g) \cdot (f \otimes x)) \\
	= & -(0 \oplus sf\xi^{(1)}) \otimes gx + (0 \oplus sf\xi^{(1)}) \otimes gx + (cf\xi^{(1)} \oplus 0) \otimes x = (cf\xi^{(1)} \oplus 0) \otimes x\ ,
\end{align*}
where we used the graded multiplication on the external tensor product $\hat{H}_{SB} \otimes \C l_1$ to obtain the second equation. Note that both of these commutators are given by compact operators and the argument for the adjoint of $M_{\xi}$ is completely analogous.

By the above computation, the operator $G$ satisfies $G^2 = 1$ (the mixed terms vanish because of the graded tensor product) and $G = G^*$. The first summand commutes with $(\phi_0 \oplus \phi_1) \otimes 1$. Therefore
\[
	[G, (\phi_0 \oplus \phi_1) \otimes 1] = \left[ 
	\begin{pmatrix} 0 & c \\ c & 0 \end{pmatrix},
	\begin{pmatrix} \phi_0 & 0 \\ 0 & \phi_1 \end{pmatrix}
	\right] \otimes 1 = \begin{pmatrix}
		0 & c(\phi_1 - \phi_0) \\
		c(\phi_0 - \phi_1) & 0
	\end{pmatrix} \otimes 1
\]
is a compact operator. This shows that \eqref{eqn:Kasparov_product} is a Kasparov module. Finally, the graded commutator between the operator for the class $[\phi_0, \phi_1]$ and $G$ evaluates to 
\begin{align*}
	& \left[\begin{pmatrix} 0 & 1 \\ 1 & 0 \end{pmatrix} \otimes 1, G\right] \\
	=\ & \begin{pmatrix} 0 & s \\ s & 0 \end{pmatrix} \otimes g - \begin{pmatrix} 0 & s \\ s & 0 \end{pmatrix} \otimes g + \begin{pmatrix} 2c & 0 \\ 0 & 2c \end{pmatrix} \otimes 1 = \begin{pmatrix} 2c & 0 \\ 0 & 2c \end{pmatrix} \otimes 1 \geq 0
\end{align*}
By \cite[Def.~18.4.1]{B} the Kasparov module defined in \eqref{eqn:Kasparov_product} represents $[\phi_0, \phi_1] \widehat{\otimes} (\mathbf{b} \otimes I_B)$.
\end{proof}

As pointed out in \cite[Cor.~14.5.3]{B} there is a $*$-isomorphism 
\begin{equation} \label{eqn:mat_Cl1_iso}
	\hat{M}_2(\C) \otimes \C l_1 \cong M_2(\C) \oplus M_2(\C)\ .
\end{equation}
It maps an even element $T \otimes 1$ to $(T,T)$, an odd element $T \otimes 1$ to $(T,-T)$ and the odd element $1 \otimes g$ to $(T_g,-T_g)$ where $T_g$ is the grading operator  
\[
	T_g = \begin{pmatrix}
		1 & 0 \\
		0 & -1
	\end{pmatrix}\ .
\]

\begin{lem}
	Let $\theta(x) = \frac{\pi}{4}(x + 1)$, $s_{\theta}(x) = \sin(\theta(x))$ and $c_{\theta}(x) = \cos(\theta(x))$. In the Cuntz picture the $KK$-class $[\phi_0, \phi_1] \widehat{\otimes} (\mathbf{b} \otimes I_B) \in KK^1(A,SB)$ is represented by the pair $(\psi, P)$ with 
	\[
		\psi = \begin{pmatrix}
		s_\theta^2\,\phi_0  + c_\theta^2\,\phi_1  & s_\theta c_\theta\,(\phi_1 - \phi_0) \\
		s_\theta c_\theta\,(\phi_1 - \phi_0) & c_\theta^2\,\phi_0  + s_\theta^2\,\phi_1 
	\end{pmatrix} \quad \text{and} \quad 
	P = \begin{pmatrix}
		1 & 0 \\
		0 & 0
	\end{pmatrix}\ .
	\]
	In particular, the Busby invariant of the associated extension is 
	\[
		\tau_{[\psi, P]}(a) = \pi(t\,\phi_0(a)  + (1-t)\,\phi_1(a))\ .
	\]
\end{lem}

\begin{proof}
The isomorphism on $\hat{M}_2(\C) \otimes M(SB \otimes \bK) \otimes \C l_1$ induced by \eqref{eqn:mat_Cl1_iso} maps $G$ to the operator $(T, -T)$ with 
\[
	T = \begin{pmatrix}
		s & c \\
		c & -s
	\end{pmatrix}\ .
\]
Let $\bar{P} = \frac{1}{2}(T + 1)$. Then we have 
\[
	\bar{P} = \frac{1}{2} \begin{pmatrix}
		(s+1) & c \\
		c & -(s-1)
	\end{pmatrix} = 
	\begin{pmatrix}
		s_{\theta} & -c_{\theta} \\
		c_{\theta} & s_{\theta} 
	\end{pmatrix}
	\begin{pmatrix}
		1 & 0 \\
		0 & 0 
	\end{pmatrix}
	\begin{pmatrix}
		s_{\theta} & c_{\theta} \\
		-c_{\theta} & s_{\theta} 
	\end{pmatrix}\ .
\]
Therefore the Kasparov product corresponds to the Cuntz pair $(\phi_0 \oplus \phi_1, \bar{P})$  Conjugating $\bar{P}$ and $\phi_0 \oplus \phi_1$ by the inverse of the rotation matrix gives $P$ and 
\[
	\begin{pmatrix}
		s_\theta & c_\theta \\
		-c_\theta & s_\theta 
	\end{pmatrix}
	\begin{pmatrix}
		\phi_0 & 0 \\
		0 & \phi_1
	\end{pmatrix}
	\begin{pmatrix}
		s_\theta & -c_\theta \\
		c_\theta & s_\theta 
	\end{pmatrix} = \begin{pmatrix}
		\phi_0 s_\theta^2 + \phi_1 c_\theta^2 & (\phi_1 - \phi_0)s_\theta c_\theta \\
		(\phi_1 - \phi_0)s_\theta c_\theta & \phi_0 c_\theta^2 + \phi_1 s_\theta^2
	\end{pmatrix}\ .
\]
Note that $t = s_{\theta}^2$ is a homeomorphism between $(-1,1)$ and $(0,1)$. The pullback of the Cuntz pair with respect to $s_{\theta}^2$ will therefore have $t\phi_0 + (1-t)\phi_1$ in the upper left hand corner. This shows the final statement. 
\end{proof}

\end{document}